\documentclass{aart}

\usepackage{titlesec}
\titleformat{\subsection}[runin]{\normalfont\bfseries}{\thesubsection.}{.5em}{}[.]\titlespacing{\subsection}{0pt}{2ex plus .1ex minus .2ex}{.8em}
\titleformat{\subsubsection}[runin]{\normalfont\itshape}{\thesubsubsection.}{.3em}{}[.]\titlespacing{\subsubsection}{0pt}{1ex plus .1ex minus .2ex}{.5em}

\usepackage[labelfont=sc,font=small,labelsep=period]{caption}
\usepackage[letterpaper, hmargin=1in, top=1.5in, bottom=1.9in, footskip=0.7in]{geometry}

\usepackage{amsmath} %[intlimits]
\usepackage{amssymb}
\usepackage{amsfonts}
\usepackage{latexsym}
\usepackage{amsthm}
\usepackage{amsxtra}
\usepackage{amscd}
\usepackage{bbm}
\usepackage{mathrsfs}
\usepackage{bm}

\usepackage{graphicx, color}

\newcommand{\f}[1]{\bm{\mathrm{#1}}} %bold
\newcommand{\bb}{\mathbb} %blackboard bold
\renewcommand{\cal}{\mathcal}

\newcommand{\ol}[1]{\overline{#1} \!\,} %overline
\newcommand{\wh}{\widehat}
\newcommand{\wt}{\widetilde}

\newcommand{\me}{\mathrm{e}}
\newcommand{\ii}{\mathrm{i}}
\newcommand{\dd}{\mathrm{d}}
\newcommand{\col}{\mathrel{\mathop:}}

\newcommand{\deq}{\mathrel{\mathop:}=}
\newcommand{\eqd}{=\mathrel{\mathop:}}
\renewcommand{\le}{\leqslant}
\renewcommand{\leq}{\leqslant}
\renewcommand{\ge}{\geqslant}
\renewcommand{\geq}{\geqslant}

\newcommand{\ind}[1]{\f 1 (#1)}

\renewcommand{\epsilon}{\varepsilon}

\renewcommand{\P}{\mathbb{P}}
\newcommand{\E}{\mathbb{E}}
\newcommand{\R}{\mathbb{R}}

\newcommand{\N}{\mathbb{N}}
\newcommand{\Z}{\mathbb{Z}}

\newcommand{\pb}[1]{\bigl({#1}\bigr)}
\newcommand{\pB}[1]{\Bigl({#1}\Bigr)}
\newcommand{\pbb}[1]{\biggl({#1}\biggr)}
\newcommand{\pBB}[1]{\Biggl({#1}\Biggr)}

\newcommand{\qB}[1]{\Bigl[{#1}\Bigr]}
\newcommand{\qbb}[1]{\biggl[{#1}\biggr]}
\newcommand{\qBB}[1]{\Biggl[{#1}\Biggr]}

\newcommand{\hb}[1]{\bigl\{{#1}\bigr\}}

\newcommand{\hbb}[1]{\biggl\{{#1}\biggr\}}

\newcommand{\abs}[1]{\lvert #1 \rvert}
\newcommand{\absb}[1]{\bigl\lvert #1 \bigr\rvert}
\newcommand{\absB}[1]{\Bigl\lvert #1 \Bigr\rvert}
\newcommand{\absbb}[1]{\biggl\lvert #1 \biggr\rvert}

\newcommand{\norm}[1]{\lVert #1 \rVert}
\newcommand{\normb}[1]{\bigl\lVert #1 \bigr\rVert}

\newcommand{\normbb}[1]{\biggl\lVert #1 \biggr\rVert}
\newcommand{\normBB}[1]{\Biggl\lVert #1 \Biggr\rVert}

\newcommand{\avg}[1]{\langle #1 \rangle}

\newcommand{\avgbb}[1]{\biggl\langle #1 \biggr\rangle}

\DeclareMathOperator{\tr}{Tr}

\DeclareMathOperator{\re}{Re}
\DeclareMathOperator{\im}{Im}

\theoremstyle{plain} %plain, definition, remark
\newtheorem{theorem}{Theorem}[section]
\newtheorem*{theorem*}{Theorem}
\newtheorem{lemma}[theorem]{Lemma}
\newtheorem*{lemma*}{Lemma}
\newtheorem{corollary}[theorem]{Corollary}
\newtheorem*{corollary*}{Corollary}
\newtheorem{proposition}[theorem]{Proposition}
\newtheorem*{proposition*}{Proposition}
\newtheorem{definition}[theorem]{Definition}
\newtheorem*{definition*}{Definition}
\theoremstyle{definition} %plain, definition, remark

\newtheorem*{example*}{Example}
\newtheorem{remark}[theorem]{Remark}

\newtheorem*{remark*}{Remark}
\newtheorem*{remarks*}{Remarks}

\newtheorem*{convention*}{Convention}

\usepackage{cite}

\newcommand{\beqa}{\begin{eqnarray}}
\newcommand{\eeqa}{\end{eqnarray}}
\newcommand{\e}{\varepsilon}

\newcommand{\rd}{{\rm d}}
\newcommand{\bR}{{\mathbb R}}

\newcommand{\bT}{{\mathbb T}}
\newcommand{\bZ}{{\mathbb Z}}
\newcommand{\non}{\nonumber}

\renewcommand{\Re}{\mbox{Re}}

\newcommand{\bv}{{\bf{v}}}

\newcommand{\bS}{\bf{S}}

\newcommand{\al}{\alpha}

\newcommand{\be}{\begin{equation}}
\newcommand{\ee}{\end{equation}}

\newcommand{\cE}{{\mathcal E}}

\newcommand{\ov}{\overline}

\newcommand{\sdot}{\cdot}

\numberwithin{equation}{section}
\numberwithin{theorem}{section}

\newcommand{\Tdet}{\Theta}%{{T^{\mbox{\tiny (det)}}}}
\newcommand{\Pdet}{\theta}
\newcommand{\Adet}{\cal A}

\title{Delocalization and Diffusion Profile for Random Band Matrices}

\author{
L\'aszl\'o Erd\H os${}^1$\thanks{Partially supported
by SFB-TR 12 Grant of the German Research Council} \quad
Antti Knowles${}^2$\thanks{Partially supported by NSF grant DMS-0757425} \quad
Horng-Tzer Yau${}^3$\thanks{Partially supported
by NSF grants  DMS-0804279 and Simons  Investigator Award }  \quad
Jun Yin${}^4$\thanks{Partially supported by NSF grants DMS-1001655 and DMS-1207961} \\\\\\
\normalsize{Institute of Mathematics, University of Munich,
Theresienstrasse 39, 80333 Munich, Germany} \\ 
\normalsize{lerdos@math.lmu.de ${}^1$} \\ \\
\normalsize{Courant Institute, New York University, 251 Mercer Street, New York, NY 10012, USA} \\
\normalsize{knowles@cims.nyu.edu ${}^2$} \\ \\
\normalsize{Department of Mathematics, Harvard University, Cambridge MA 02138, USA} \\
\normalsize{htyau@math.harvard.edu ${}^3$}  \\  \\
\normalsize{Department of Mathematics, University of Wisconsin, Madison, WI 53706, USA} \\
\normalsize{jyin@math.uwisc.edu ${}^4$} \\ \\ 
}

\begin{document}

\maketitle

\begin{abstract}
We consider Hermitian and symmetric random band matrices $H = (h_{xy})$ in $d \geq 1$ dimensions.
The matrix entries $h_{xy}$, indexed by $x,y \in (\bZ/L\bZ)^d$, are
independent, centred random variables with variances $s_{xy} = \E |h_{xy}|^2$. We assume 
that $s_{xy}$ is negligible if $|x-y|$ exceeds the band width $W$.
In one dimension we prove that the eigenvectors of $H$ are delocalized if $W\gg L^{4/5}$.
We also show that the magnitude of the matrix entries $\abs{G_{xy}}^2$ 
of the resolvent $G=G(z)=(H-z)^{-1}$ is self-averaging and we compute $\E \abs{G_{xy}}^2$.
We show that, as $L\to\infty$ and $W\gg L^{4/5}$, the behaviour
of $\E |G_{xy}|^2$ 
 is governed by a diffusion operator whose diffusion constant we compute. Similar results are obtained in higher dimensions.
\end{abstract}

% \vspace{1cm}
% {\bf AMS Subject Classification:} 15B52, 82B44, 82C44

\vspace{0.5cm}


\newpage

\section{Introduction}

Random band matrices $H=(h_{xy})_{x,y\in \Gamma}$ represent  quantum systems on a large finite graph 
$\Gamma$ with random quantum transition amplitudes effective up to distances of order $W$.
The matrix entries are independent, centred
random variables. The variance $s_{xy} \deq \E |h_{xy}|^2$ depends on the distance 
between the two sites $x$ and $y$, and it typically decays with the distance on a characteristic 
length scale $W$, called the \emph{band width} of $H$. 
This terminology comes from the simplest one-dimensional
model where the graph $\Gamma = \{1, 2, \dots, N\}$ is a path on $N$ vertices, and the matrix entries $h_{xy}$
are negligible if $|x-y|\geq W$. In particular, if $W=N$ and all
variances are equal, we recover the well-known Wigner matrix, which corresponds to a mean-field model.
Higher-dimensional models 
are obtained if $\Gamma$ taken to be the box 
of linear size $L$  in $\Z^d$. In this case the dimension of the matrix is
$N=L^d$.

Typically, $W$ is a mesoscopic scale, larger than the lattice spacing but smaller
than the diameter $L$ of the system: $1\ll W\ll L$.
These models are natural interpolations between random Schr\"odinger
operators with short range quantum transitions such as the Anderson model \cite{And}
and mean-field random matrices such as Wigner matrices \cite{Wig}. In particular,
random band matrices may be used to model the \emph{Anderson metal-insulator phase transition}, which we briefly outline.

The key physical parameter of all these models is the \emph{localization length} $\ell$,
which describes the typical length scale of the eigenvectors of $H$.
 The system is said to be {\it delocalized}
 if the localization length is comparable with the system size, $\ell\sim L$,
and it is {\it localized} otherwise.  Delocalized systems are electric conductors,
while localized systems are insulators.

Nonrigorous supersymmetric calculations \cite{FM} show that for random band matrices
the localization length is of order $\ell\sim W^2$ in $d=1$ dimension.
In $d=2$, the localization length is expected to be exponentially growing in $W$,
and in $d\ge 3$, it is macroscopic, $\ell\sim L$, i.e.\
the system is delocalized. We refer to the overview papers of Spencer
\cite{Spe, Sp} and to the paper of Schenker \cite{Sch} for more details on these conjectures.

These predictions are in accordance with those for the \emph{Anderson model}, where the random matrix is of the form $-\Delta+\lambda V$; here $\Delta$ is the lattice Laplacian, $V$ a random potential (i.e.\ a diagonal matrix with i.i.d.\ entries), and $\lambda$ a small coupling constant.
The localization length is $\ell\sim \lambda^{-2}$ in 
the regime of strong localization, which corresponds to the whole spectrum for $d=1$ and a neighbourhood of the spectral edges for $d > 1$.
This result follows from the rigorous multiscale analysis
of Fr\"ohlich and Spencer \cite{FroSpe} as well as 
from the fractional moment method of Aizenman and Molchanov \cite{AizMol}.
The two-dimensional Anderson model is conjectured
to be in the weak localization regime with $\ell\sim \exp (\lambda^{-2})$
throughout the spectrum \cite{Abr}, but this has so far not been proved.
In dimensions $d\ge 3$, the prediction is that there is a threshold energy, called the mobility
edge, $E_0$, that separates the localized regime
near the band edges from the delocalized regime in the bulk spectrum.
The localization length is expected to diverge as the energy $E$ approaches
the mobility edge from the localization side. The increase of the
localization length as an inverse power of $E-E_0$ has been
rigorously established up to a certain scale in \cite{Spen, Elg},
but this analysis does not allow $E$ to actually reach the conjectured value
of $E_0$.
A  key  open question for the Anderson model is to establish the metal-insulator
transition, i.e.\ to show that the mobility edge indeed exists.

For random band matrices, the metal-insulator transition can be investigated
even in $d=1$ by varying the band width $W$. 
The prediction that the localization
length $\ell$ is of order $W^2$ can be recast in the form that the
eigenvectors are delocalized if $W \ge L^{1/2}$.
Currently only lower and upper bounds have been established for $\ell$.
On the side of localization, Schenker \cite{Sch} proved that $\ell\le W^8$,
uniformly in the system size,
by extending the methods of the proofs of the Anderson localization for random Schr\"odinger operators.
As a lower bound, $\ell\ge W$ was proved in \cite{EYY1} by using
a self-consistent equation for the diagonal matrix entries $G_{xx}$
of the Green function $G=G(z)=(H-z)^{-1}$.
In particular, Wigner matrices ($W=L$) are completely delocalized; in fact
this has been proven earlier in \cite{ESY1, ESY2, ESY4} using a simpler self-consistent
equation for the trace of the Green function, $\tr G$.
This lower bound was improved to $\ell \ge W^{7/6}$ in \cite{EK1, EK2}
by using diagrammatic perturbation theory.
In fact, not only was the lower bound on localization length established,
but it was also shown that the unitary time evolution, $\me^{\ii tH}$,
behaves diffusively on the spatial scale $W$, i.e.\ the typical propagation
distance is $\sqrt{t} W$. Thus, the mechanism responsible for the delocalization of random band matrices is a random walk (in fact, a superposition of random walks) with step size of order $W$. Showing that the localization length is greater than the naive size $W$ therefore requires a control of the random walk for large times. For technical reasons, in \cite{EK1, EK2} the time evolution could only be controlled up to time $t\le W^{1/3}$, which corresponds to delocalization on the scale $W^{1 + 1/6}$. The work \cite{EK1,EK2} was partly motivated by 
a similar results for the long-time evolution for the Anderson model
\cite{ESalY1, ESalY2, ESalY3} combined with an algebraic renormalization
using Chebyshev polynomials \cite{FS, So1}.
 
In the current paper we develop a new self-consistent equation (see \eqref{Tself}) which
keeps track of all matrix entries of the Green function $G$, and not only the
diagonal ones as in \cite{EYY1}. We 
show that $|G_{xy}|^2$ is self-averaging and $\E |G_{xy}|^2$ 
behaves as the resolvent of a diffusion operator associated with a superposition of random walks with step size of order $W$. This result can then be translated into
a lower bound on the localization length.

More precisely, for $d=1$ we obtain full control on $|G_{xy}(z)|^2$
for relatively broad bands, $W \gg N^{4/5}$, and 
for $\eta =\im z \ge (W/N)^2$. The condition 
$W\gg N^{4/5}$ is technical.  The condition on $\eta$ comes from the facts that 
$t = \eta^{-1}$ corresponds to the time scale of the random walk, and a random walk with step size $W$ 
in a box of size $N$ reaches equilibrium in a time of order $(N/W)^2$.
As a corollary, we prove that most eigenvectors are
delocalized if $W\gg N^{4/5}$. This improves the exponent 
in  \cite{EK1,EK2}, where delocalization
for $W\gg N^{6/7}$ was proved. However, 
unlike in \cite{EK1,EK2}, here we do not obtain
a lower bound on the localization length $\ell$ uniformly in $N$.
We also prove analogous results in higher dimensions.  In addition, we investigate the case where 
the variances $s_{xy}$ of the matrix entries decay slowly according to the power law
 $\abs{x - y}^{-(1+\beta)}$ for $0<\beta<2$. In this regime the system exhibits
 superdiffusive behaviour. In particular, we may allow decay of the form $\abs{x-y}^{-2}$, which is critical in the sense of \cite{MFDQS}.

One key open question for random band matrices is to control the resolvent $G(z)$
for $\eta =\im z\ll W^{-1}$. None of the results  mentioned above
yield a nontrivial control below $W^{-1}$. In the regime $\eta \ge W^{-1}$ 
robust pointwise bounds on $G_{xy}$ have been obtained 
with high probability \cite{EYY1}.
For $\frac{1}{N} \E \tr G(z)$ and for $\eta\ge W^{-0.99}$,
a more precise error estimate was derived
for a special class of Bernoulli entries in \cite{So2}.  However,  
controlling  the quantity $\frac{1}{N} \E \tr G(z)$ 
does not yield information on the localization length. 
In the current paper we obtain much more precise  bounds on $|G_{xy}|^2$
in the regime $\eta\ge (W/N)^2\ge W^{-1/2}$ ,
which in particular imply delocalization bounds. 

Supersymmetric (SUSY) methods offer a very attractive approach to studying 
the delocalization transition in band matrices,
but rigorous control of the ensuing functional integrals away from the
saddle points is difficult. This task
has been performed for the density of states of a special
three-dimensional Gaussian model \cite{DPS}; this is the only
result where a nontrivial control for $\eta\le W^{-1}$
(in fact, uniform in $\eta$) was obtained. 
The SUSY method has so far only been applied to the expectation of single Green function,
$\E G$, and not to its square, $\E |G|^2$.

The analysis of the trace of the single Green function yields the limiting spectral density of $H$
which is the Wigner semicircle law provided the band width $W$ diverges as $L \to \infty$. For band matrices
the semicircle law on large scales, corresponding to spectral parameter $\eta>0$ independent of $N$,
 was given in \cite{MPK}.
More recently, a \emph{semicircle law on small scales}, in which $\eta \ll 1$, was derived in \cite{EYY1} 
and generalized in \cite{EKYY4}. The results of \cite{EYY1, EKYY4} are summarized in Lemma \ref{lm:lsc} below. As an application of our method, we prove a further improvement of the semiricle law in Theorem \ref{lm:noprof} below.

The main new ingredient in this paper is the self-consistent equation 
for the matrix  $T$, whose entries
$$
   T_{xy} \;\deq\; \sum_i s_{xi} |G_{iy}|^2
$$
are local averages of $|G_{xy}|^2$. We show in Theorem~\ref{thm:T} below
that $T$ satisfies a self-consistent equation of the form
\begin{equation} \label{self-const intro}
   T \;=\; |m|^2 ST+|m|^2 S + \cE\,,
\end{equation}
where $S$ is the matrix of variances $(s_{xy})$,
$m \equiv m(z)$ is an explicit function of the spectral parameter $z=E+\ii \eta$ (see \eqref{definition of msc} below),
and $\cE$ is an error term.
Neglecting the error term $\cE$, we obtain
$$
  T \;\approx\; \frac{|m|^2S}{1-|m|^2S}\,.
$$
In this paper we implement the band structure of $H$ using a symmetric probability density $f$ on $\R^d$,
 by requiring that $s_{xy} \approx W^{-d} f((i - j) / W)$ (see Section \ref{sec:setup} below for the precise statement).
Using translation invariance of $S$ and the Taylor expansion of its Fourier transform $\wh S(p)$
in the low momentum regime, we obtain for $\abs{p} \ll W^{-1}$ that
\begin{equation} \label{low-p exp}
\wh S(p) \;\approx\; 1 - W^2 (p \sdot D p) + \cdots\,,
\end{equation}
where $D$ is the matrix of second moments of $f$ (see \eqref{Dd} below).
In order to give the leading-order behaviour of $T$, we
use  $|m|^2= 1-\alpha\eta+ O(\eta^2)$ (see \eqref{m2} below), where
\begin{equation} \label{def alpha}
\alpha \;\equiv\; \alpha(E) \;\deq\; \frac{2}{\sqrt{4 - E^2}} \qquad (E = \re z)\,.
\end{equation}
Therefore the Fourier transform of $T$ is approximately given by
\begin{equation} \label{approx T hat}
    \frac{\alpha^{-1}}{\eta +  W^2(p \sdot D_{\rm eff} \, p)} \qquad \text{where} \qquad  D_{\rm eff} \;\deq\;
  \frac{D}{\alpha}\,,
\end{equation}
in the regime $\abs{p}\ll W^{-1}$  and $\eta\ll 1$. 
This corresponds to the  diffusion approximation on scales larger than $W$
with  an effective diffusion constant  $D_{\rm eff}$. In the language of diagrammatic perturbation theory, the change from $D$ to $D_{\rm eff}$ has the interpretation of a self-energy renormalization. This result coincides with Equation (1.5.5) of \cite{Sp}, which was obtained by computing the sum of ladder diagrams in a high-moment expansion.

The main result of this paper is a justification of this heuristic argument in a certain range of parameters. The error term $\cE$ contains fluctuations of
local averages. Roughly speaking,
we need to control the size of $\sum_x \big[ |G_{xy}|^2 - P_x|G_{xy}|^2\big]$,
 where $P_x$ denotes partial expectation
with respect to the matrix entries in the $x$-th row (see  $\wt T_{xy}$ in \eqref{fluc1T} below).
Unfortunately, $|G_{xy}|^2$ and $|G_{x'y}|^2$ for  $x\ne x'$ are not independent;
in fact they are strongly correlated for small $\eta$, and they do not behave like 
independent random variables.  Estimating high moments of these averages
requires an unwrapping of the hierarchical correlation structure among several
resolvent matrix entries. The necessary estimates are quite involved.
They are a special case of the more general Fluctuation Averaging Theorem
 that is published separately \cite{EKY2}, and 
was originally developed for application in the current paper. There have been several previous results in this direction; see 
\cite[Lemma 5.2]{EYY2},  \cite[Lemma 4.1]{EYY3}, 
 \cite[Theorem 5.6]{EKYY1}, and \cite[Theorem 3.2]{PY}.
 The Fluctuation Averaging Theorem 
generalizes these ideas to arbitrary monomials of $G$ and
exploits an additional cancellation mechanism in averages
of $|G_{xy}|^2$ that is not present in averages of $G_{xx}$.
 For more details, see \cite{EKY2}.

\section{Formulation of the results}

\subsection{Setup} \label{sec:setup}
Fix $d \in \N$ and let $f$ be a smooth and symmetric (i.e.\ $f(x) = f(-x)$) probability density on $\R^d$.
Let $L$ and $W$ be integers satisfying
\begin{equation} \label{lower bound on W}
L^\delta \;\leq\; W \;\leq\; L
\end{equation}
for some fixed $\delta > 0$. The parameter $L$ is the fundamental large quantity of our model. Define the $d$-dimensional discrete torus
\begin{equation*}
\bb T^d_L \;\deq\; [-L/2, L/2)^d \cap \Z^d\,.
\end{equation*}
Thus, $\bb T^d_L$ has $N \deq L^d$ lattice points. For the following we fix an (arbitrary) ordering of $\bb T_L^d$, which allows us to identify it with $\{1, \dots, N\}$. We define the canonical representative of $i \in \Z^d$ through
\begin{equation*}
[i]_L \;\deq\; (i + L \Z^d) \cap \bb T^d_L\,,
\end{equation*}
and introduce the periodic distance
\begin{equation*}
\abs{i}_L \;\deq\; \absb{[i]_L}\,,
\end{equation*}
where $\abs{\cdot}$ denotes Euclidean distance in $\R^d$.

Define the $N \times N$ matrix $S(L,W) \equiv S = (s_{ij} \col i,j \in \bb T_L^d)$ through
\begin{equation}\label{sij}
s_{ij} \;\deq\; \frac{1}{Z_{L,W}} \, f \pbb{\frac{[i - j]_L}{W}}\,,
\end{equation}
where $Z_{L,W}$ is a normalization constant chosen so that $S$ is a stochastic matrix:
\begin{equation}\label{stoch}
\sum_j s_{ij} \;=\; 1
\end{equation}
for all $i \in \bb T_L^d$. Unless specified otherwise, summations
are always over the set $\bb T_L^d$.
By symmetry of $f$ we find that $S$ is symmetric: $s_{ij} = s_{ji}$.
As a stochastic matrix, the spectrum of $S$ lies in $[-1,1]$.
In fact it is proved
in Lemma A.1 of \cite{EYY1} that there exists a positive constant $\delta$, depending only on $f$,
such that
\begin{equation}
\label{Slow}
   -1+\delta \;\leq\; S\le 1\,.
\end{equation}

We let $(\zeta_{ij} \col i \leq j)$, where $i,j \in \bb T_L^d$, be a family of independent, 
complex-valued, centred random variables $\zeta_{ij} \equiv \zeta_{ij}^{(N)}$ satisfying
\begin{equation}\label{zetacond}
\E \zeta_{ij} \;=\; 0\,, \qquad \E \abs{\zeta_{ij}}^2 \;=\; 1\,, \qquad \zeta_{ii}\in \bR.
\end{equation}
 For $i>j$ we define 
$$
    \zeta_{ij} \; \deq\; \bar \zeta_{ji}.
$$
We define the band matrix $H = (h_{ij})_{i,j\in \bT_L^d}$ through
\begin{equation}\label{def:band}
h_{ij} \;\deq\; (s_{ij})^{1/2} \, \zeta_{ij}\,.
\end{equation}
Thus we have $H = H^*$ and
\begin{equation}
\E \abs{h_{ij}}^2 \;=\; s_{ij}\,.
\end{equation}
In particular, we may consider the two classical symmetry classes of random matrices: 
real symmetric and complex Hermitian. For \emph{real symmetric band matrices} we assume
\begin{equation} \label{RS}
\zeta_{ij} \in \R \quad \text{for all} \quad i \leq j \,.
\end{equation}
For {\it complex Hermitian band matrices}
we assume
\begin{equation} \label{CH}
\E \zeta_{ij}^2 = 0 \quad \text{for all} \quad i < j\,.
\end{equation}
in addition  to \eqref{zetacond}.
A common way to satisfy \eqref{CH} is to choose the real and imaginary parts of $\zeta_{ij}$
 to be independent with identical variance. As in \cite{EKY2}, 
our results also hold without this assumption, but we omit the details of this generalization
to avoid needless complications.

We introduce the parameter
\begin{equation}  \label{s leq W}
M \;\equiv\; M_N \;\deq\; \frac{1}{\max_{i, j} s_{ij}}\,.
\end{equation}
From the definition of $S$ it is easy to see that $Z_{N,W} = W^d +O(W^{d-1})$. In particular,
\begin{equation*}
M \;=\; \pb{W^d + O(W^{d - 1})} / \norm{f}_\infty\,.
\end{equation*}

We assume that the random variables $\zeta_{ij}$ have finite moments, uniformly in $N$, $i$, and $j$, in the sense that for all $p \in \N$ there is a constant $\mu_p$ such that
\begin{equation} \label{finite moments}
\E \abs{\zeta_{ij}}^p \;\leq\; \mu_p
\end{equation}
for all $N$, $i$, and $j$.

The following definition introduces a notion of a high-probability bound that is suited for our purposes.
\begin{definition}[Stochastic domination]\label{sdI}
Let $X = \pb{X^{(N)}(u) \col N \in \N, u \in U^{(N)}}$ be a family of random variables, where $U^{(N)}$ is a possibly $N$-dependent parameter set. Let $\Psi = \pb{\Psi^{(N)}(u) \col N \in \N, u \in U^{(N)}}$ be a deterministic family satisfying $\Psi^{(N)}(u) \geq 0$. We say that $X$ is \emph{stochastically dominated by $\Psi$, uniformly in $u$,} if for all $\epsilon > 0$ and  $D > 0$ we have
\begin{equation*}
\sup_{u \in U^{(N)}} \P \qB{\absb{X^{(N)}(u)} > N^\epsilon \Psi^{(N)}(u)} \;\leq\; N^{-D}
\end{equation*}
for large enough $N\ge N_0(\e, D)$. Unless stated otherwise, 
throughout this paper the stochastic 
domination will always be uniform in all parameters apart from the parameter $\delta$ in \eqref{lower bound on W} and the sequence of constants $\mu_p$ in \eqref{finite moments}; thus, $N_0(\e, D)$ also depends on $\delta$ and $\mu_p$.
If $X$ is stochastically dominated by $\Psi$, uniformly in $u$, we use the equivalent notations
\begin{equation*}
X \;\prec\; \Psi \qquad \text{and} \qquad X \;=\; O_\prec(\Psi)\,.
\end{equation*}
\end{definition}
For example, using Chebyshev's inequality and \eqref{finite moments} one easily finds that 
\begin{equation}\label{hsmallerW}
h_{ij} \;\prec\; (s_{ij})^{1/2} \;\leq\; M^{-1/2}\,,
\end{equation}
so that we may also write $h_{ij} = O_\prec((s_{ij})^{1/2})$.
The relation $\prec$ satisfies the familiar algebraic rules of order relations.
The general statements are formulated later in Lemma~\ref{lemma: basic properties of prec}.

We remark that Definition \ref{sdI} is tailored to the assumption that \eqref{finite moments}
holds for any $p$. If  \eqref{finite moments} only holds for some large but fixed $p$ then all of our results still hold, but in a somewhat weaker sense. Indeed,
the control of the exceptional events in our theorems 
is expressed via the relation $\prec$. If only finitely many
moments are assumed to be finite in \eqref{finite moments}, then
  the exponents $\e$ and $D$ in the definition of $\prec$
cannot be chosen to be arbitrary, and will in fact depend on $p$.
Repeating our arguments under this weaker assumption would require us to follow all of these exponents
through the entire proof. Our assumption that \eqref{finite moments}
holds for any $p$ streamlines our statements and proofs, by avoiding the need to keep track of the precise values of these parameters.

Throughout the following we make use of a spectral parameter
\begin{equation*}
z \;=\; E + \ii \eta \,, \qquad E \in \R, \qquad \eta > 0\,.
\end{equation*}
We choose and fix two arbitrary (small) global constants $\gamma > 0$ and $\kappa > 0$.
 All of our estimates will depend on $\kappa$ and $\gamma$, and we shall often 
omit the explicit mention of this dependence. Set
\begin{equation}\label{def:S}
\f S \;\equiv\; \f S^{(N)}(\kappa, \gamma) \;\deq\; \hb{E + \ii \eta \col -2 + \kappa \leq E \leq 2 - \kappa \,,\, M^{-1 + \gamma} \leq \eta \leq 10}\,.
\end{equation}
We shall always assume that the spectral parameter $z$ lies in $\f S(\kappa, \gamma)$. In this paper we always consider families $X^{(N)}(u) = X^{(N)}_i(z)$ indexed by $u = (z,i)$, where $z \in \f S(\kappa,\gamma)$ and $i$ takes on values in some finite (possibly $N$-dependent or empty) index set.

We introduce the Stieltjes transform of Wigner's semicircle law, defined by
\begin{equation} \label{definition of msc}
m(z) \;\deq\; \frac{1}{2 \pi} \int_{-2}^2 \frac{\sqrt{4 - \xi^2}}{\xi - z} \, \dd \xi\,.
\end{equation}
 It is well known that the Stieltjes transform $m$ is characterized by the unique solution of
\begin{equation} \label{identity for msc}
m(z) + \frac{1}{m(z)} + z \;=\; 0
\end{equation}
with $\im m(z) > 0$ for $\im z > 0$. Thus we have
\begin{equation} \label{explicit m}
m(z) \;=\; \frac{-z + \sqrt{z^2 - 4}}{2}\,.
\end{equation}
To avoid confusion, we remark that the Stieltjes transform $m$ was denoted by $m_{sc}$ 
in the papers \cite{ESY1, ESY2, ESY3, ESY4,ESY5,ESY6, ESY7, ESYY, EYY1, EYY2, EYY3, EKYY1, EKYY2}, in which $m$ had a different meaning from \eqref{definition of msc}. 

We define the \emph{resolvent} of $H$ through
\begin{equation*}
G \;\equiv\; G(z) \;\deq\; (H - z)^{-1}\,,
\end{equation*}
and denote its entries by $G_{ij}(z)$. In the following sections we list our main results on the resolvent matrix entries.

We conclude this section by introducing some notation that will be used throughout the paper. We use $C$ to denote a generic large positive constant, which may depend on some fixed parameters and whose value may change from one expression to the next. Similarly, we use $c$ to denote a generic small positive constant. For two positive quantities $A_N$ and $B_N$ we sometimes use the notation $A_N \asymp B_N$ to mean $c A_N \leq B_N \leq C A_N$. Moreover, we use $A_N \ll B_N$ to mean that there exists a constant $c > 0$ such that $A_N \leq N^{-c} B_N$; we also use $A_N \gg B_N$ to denote $B_N \ll A_N$. (Note that these latter conventions are nonstandard.) Finally, we introduce the Japanese bracket $\avg{x} \deq \sqrt{1 + \abs{x}^2}$. Most quantities 
in this paper depend on the spectral parameter $z$, which we however mostly omit from the notation.

For simplicity, here we state our main results assuming that $d = 1$ and that $f$ satisfies the decay condition
\begin{equation} \label{decay of f}
\abs{f(x)} \;\leq\; C_n \avg{x}^{-n} \qquad \text{for all } n \in \N\,.
\end{equation}
Since $d = 1$, we have $N = L$ and we shall consistently use $N$ instead of $L$. Similarly, $M \asymp W$, and we shall consistently use $W$ in estimates. We also abbreviate $\bb T_N^1 \equiv \bb T$.

The generalization of our results to $d > 1$ and slowly decaying $f$ is straightforward, and will be given in Section~\ref{sec:gen}. We emphasize that
the core of our argument, given in Sections~\ref{sec:prelim}--\ref{sec:inv}, is valid in general, independent of the dimension.

\subsection{Improved local semicircle law for resolvent entries and delocalization}
Throughout this section we assume $d = 1$ and \eqref{decay of f}.
The Wigner semicircle law states that the normalized trace,  $\frac{1}{N}\tr G(z)$,
is asymptotically given by $m(z)$. In fact, this
asymptotics holds even for individual matrix entries. Our first
theorem controls the ($z$-dependent) random variable
\begin{equation*}
\Lambda(z) \;\deq\; \max_{x,y} \absb{G_{xy}(z) - \delta_{xy} m(z)}\,.
\end{equation*}
For the following we introduce the deterministic control parameter $\Phi \equiv \Phi^{(N)}(z)$ through
\begin{equation} \label{def Phi}
\Phi^2 \;\deq\; \max \hbb{\frac{1}{N\eta}, \frac{1}{W\sqrt{\eta}}}\,.
\end{equation}

\begin{theorem}[Improved local semicircle law]\label{lm:noprof}
Assume $d = 1$ and \eqref{decay of f}. Suppose moreover that
\begin{equation} \label{cond on N eta}
N \ll W^{5/4}\,, \qquad \eta \gg N^{2} / W^3\,.
\end{equation}
Then we have
\be\label{decc}
   \Lambda^2 \;\prec\; \Phi^2
\ee
for $z\in \f S$.
\end{theorem}
Clearly, the assumption $\eta \gg N^{2}/ W^3$ can be replaced with the stronger assumption $\eta \gg W^{-1/2}$. The assumption $N \ll W^{5/4}$ is technical; to see why it is needed, see \eqref{init} in the proof of Theorem \ref{lm:noprof} below. In the regime \eqref{cond on N eta}, Theorem \ref{lm:noprof} improves the earlier result 
\be\label{bandbound}
\Lambda^2 \;\prec\; \frac{1}{M\eta}
\ee
proved in \cite{EYY1} (see Lemma~\ref{lm:lsc} below).
In fact, the estimate \eqref{decc} is optimal, as may be seen from \eqref{supexpl} and the first estimate of \eqref{Tprec} below. By spectral decomposition of $G$ one easily finds that
$$
  \frac{1}{N^2}  \sum_{x,y} |G_{xy}|^2 \;=\; \frac{1}{N^2} \tr G^* G \;=\; \frac{1}{N\eta} \im \frac{\tr G}{N} 
   \;=\; \frac{\im m}{N\eta} + O_\prec \pbb{ \frac{\Lambda}{N\eta}}\,.
$$
Thus, in the regime where $\Lambda$ is bounded, the average of $\abs{G_{xy}}^2$ is of order $(N\eta)^{-1}$. Here we introduced the notation $G^*(z) \deq (G(z))^* = (H - \bar z)^{-1}$, which we shall use throughout the following.

\begin{remark}
The bound \eqref{decc} implies
 an estimate on the Stieltjes transform of the empirical spectral density, 
$m_N(z) \deq N^{-1} \tr G(z)$. Under the assumptions of Theorem \ref{lm:noprof} and the 
conditions \eqref{cond on N eta}, we have 
\be\label{m-mPhi}
m_N(z) - m(z) \prec \Phi^2
\ee
for $z\in  \f S$. Once $\Lambda\prec \Phi$ is established, \eqref{m-mPhi} easily follows from 
\be\label{simpleFA}
  \frac{1}{N} \sum_k Q_k G_{kk} \;\prec\; \Phi^2\,;
\ee
we leave the details to the reader. We remark that \eqref{simpleFA}
is the simplest form of the fluctuation averaging mechanism
(see Section~\ref{sec:Av}). A concise proof of  \eqref{simpleFA} can be 
found in \cite[Theorem 4.6]{EKYY4}.
\end{remark}

For $\eta \le (W/N)^2$ we have $\Phi^2= (N\eta)^{-1}$,
and the bound \eqref{decc} therefore shows that all off-diagonal entries of $G$ have a magnitude comparable with the average of their magnitudes.
We say that the resolvent is \emph{completely delocalized}.
Complete delocalization of the resolvent implies that the eigenvectors are
completely delocalized in a weak sense. The precise formulation is given 
in Proposition~\ref{lemma: simple deloc} below. 
By choosing $\eta$ such that $W^{-1/2}\le \eta\le (W/N)^2$
and invoking Proposition~\ref{lemma: simple deloc}
we obtain the following corollary.

\begin{corollary}[Eigenvector delocalization]\label{cor:noprof}
Assume $d = 1$ and \eqref{decay of f}. If  $ N \ll W^{5/4}$ then the eigenvectors of $H$ are completely delocalized in the sense of Proposition \ref{lemma: simple deloc} below.
\end{corollary}
This corollary improves the result in \cite{EK1, EK2}, where complete eigenvector delocalization (in a slightly weaker sense; see Remark \ref{rem: old result} below) was proved under the condition
$N \ll W^{7/6}$. It was observed in Section 11 of \cite{EK1} that the graphical perturbative renormalization scheme of \cite{EK1, EK2} faces a fundamental barrier at $N = W^{6/5}$. The reason for this barrier is that a large family of graphs whose contribution was subleading for $N \ll W^{6/5}$ in fact yield a leading-order contribution for $N \geq W^{6/5}$
if estimated individually. The cancellation mechanism among
these subleading graphs has so far not been identified.
As evidenced by Corollary \ref{cor:noprof}, our present  approach 
goes beyond this barrier.

\subsection{Diffusion profile}
In this section we assume $d = 1$ and \eqref{decay of f}.
In the previous section we saw that for $\eta\le (W/N)^2$ the profile of $|G_{xy}|^2$
 is essentially flat.
In the complementary regime, $\eta\ge (W/N)^2$, the 
 averaged resolvent $\E \abs{G_{xy}}^2$ is determined by a non-constant
deterministic profile given by the diffusion approximation
\be\label{Ydet}
\Tdet_{xy} \;\deq\; \pbb{\frac{\abs{m}^2 S}{1 - \abs{m}^2 S}}_{xy} \qquad (x,y \in \bb T)\,.
\ee
Note that the matrix $\Tdet = (\Tdet_{xy})$ solves the equation
\begin{equation*}
\Tdet \;=\; |m|^2 S \Tdet + |m|^2S\,,
\end{equation*}
which is obtained from \eqref{self-const intro} by dropping the error term $\cE$.
Clearly, $\Tdet_{xy}$ is translation invariant, i.e.\ $\Tdet_{xy}=\Tdet_{u0}$ with $u=[x-y]_N$.
Moreover,  $\Tdet_{xy} > 0$ for all $x,y$. Indeed, this follows immediately from the geometric series representation
\begin{equation} \label{random walk picture}
\Tdet_{xy} \;=\; \sum_{n \geq 1} \abs{m}^{2n} (S^{n})_{xy}\,,
\end{equation}
which converges by $|m|<1$ (see \eqref{m3} below)
and the trivial bound $0 \leq (S^n)_{xy} \leq 1$, as follows from \eqref{stoch}.

The representation \eqref{random walk picture} in fact provides the following interpretation of $\Tdet_{xy}$ in terms of random walks. 
From \eqref{m2} below we find that $\abs{m}^2 \approx \me^{-  \alpha \eta}$
(recall the definition of $\alpha$ from \eqref{def alpha}).
Thus the right-hand side of \eqref{random walk picture} may be approximately written as $\sum_{n \geq 1} \me^{-n \alpha \eta} (S^n)_{xy}$. By definition, $S$ is a doubly stochastic matrix -- the transition matrix of a random walk on $\bb T$ whose steps are of size $W$ and whose transition probabilities are given by $p(x \to y) = s_{xy}$. The normalized variance of each step is given by the \emph{unrenormalized diffusion constant}
\begin{equation} \label{def D}
D \;\equiv\; D_W \;\deq\; \frac{1}{2} \sum_{u \in \bb T} \pbb{\frac{u}{W}}^2 s_{u0}\,.
\end{equation}
(We normalize by $W^{-2}$ to account for the fact that the distribution $s_{u0}$ has variance $O(W^2)$.)
It is easy to see that
\begin{equation} \label{Dinfty}
D \;=\; D_\infty + O(W^{-1}) \qquad \text{where} \qquad D_\infty \;\deq\; \frac{1}{2} \int x^2 f(x) \, \dd x\,.
\end{equation}
We conclude that $\Tdet_{xy}$ is a superposition of random walks up to times of order $(\alpha \eta)^{-1}$. In this superposition the random walk with $n$ steps carries a weight $\abs{m}^{2n} \approx \me^{-n\alpha \eta}$, so that walks with times larger than $(\alpha \eta)^{-1}$ are strongly suppressed.
 The total weight of $\Tdet_{u0}$ is 
\begin{equation} \label{total mass 0}
\sum_u \Tdet_{u0} \;=\; \sum_{n \geq 1} \abs{m}^{2n} \;\approx\; (\alpha \eta)^{-1}\,;
\end{equation}
a precise computation is given in \eqref{total mass 3} below.

The following theorem shows that an averaged version of $|G_{xy}|^2$ 
is asymptotically given by $ \Tdet_{xy}$ with high probability. The averaging can be done in two ways. First, we can take the expectation $\E \abs{G_{xy}}^2$. In fact, taking \emph{partial expectation} $P_x \abs{G_{xy}}^2$ is enough; here $P_x$ denotes partial expectation in the randomness of the $x$-th row of $H$ (see Definition \ref{definition: P Q} below). Second, we can average in the index $x$ (or $y$ or both) on a scale of $W$; for simplicity we consider the weighted average
\be\label{def:T}
    T_{xy} \;\deq\; \sum_{i} s_{xi}|G_{iy}|^2\,.
\ee
Note that $T$ is not symmetric, but our results also hold for $T_{xy}$ replaced with the quantities $\sum_{j} s_{yj} \abs{G_{xj}}^2$ or $\sum_{i,j} s_{xi} s_{yj} \abs{G_{ij}}^2$.

\begin{theorem}[Diffusion profile]\label{lm:withprof}
Assume $d = 1$ and \eqref{decay of f}. Suppose that $ N \ll W^{5/4}$ and $(W/N)^2 \le \eta \le 1$.
Then
\be\label{Tprec}
 |T_{xy} -\Tdet_{xy}| \;\prec\; \frac{1}{N\eta}\,, \qquad
  \absB{P_x |G_{xy}|^2 - \delta_{xy}|m|^2- \abs{m}^2 \Tdet_{xy}} \;\prec\; \frac{1}{N\eta}
 + \frac{\delta_{xy}}{\sqrt{W}}\,.
\ee
In addition, we have the upper bounds
\be\label{Tfin}
  T_{xy} \;\prec\;  \Upsilon_{xy}  
\ee
and
\be\label{Tfin1}
\absb{G_{xy}-\delta_{xy} m}^2 \;\prec\;  \Upsilon_{xy}\,,
\ee
where we defined
\be\label{Tdetbound}
\Upsilon_{xy}\; \equiv \; \Upsilon_{xy}^{(K)} \;\deq\; \frac{1}{N\eta} +\frac{1}{W\sqrt{\eta}}
\exp \qbb{ -\frac{\sqrt{\al\eta}}{W\sqrt{D}}|x-y|_N} 
   +\frac{1}{W} \avgbb{ \frac{\sqrt{\eta}|x-y|_N}{W}}^{-K}\,.
\ee
Here  $K$ is an arbitrary, fixed, positive integer.
All estimates are uniform in $z\in \f S$ and $x,y \in \bb T$.
\end{theorem}

Note that the total mass of the distribution $\abs{G_{x0}}^2$ may be computed explicitly by spectral decomposition of $G$: assuming $\Lambda \prec \Psi$ we have
\begin{equation} \label{total mass 2}
\sum_{x} T_{x0} \;=\; \sum_x \abs{G_{x0}}^2 \;=\;  \frac{\im G_{00}}{\eta} \;=\; \frac{\im m}{\eta} (1 + O_\prec(\Psi))\,,
\end{equation}
in agreement with the corresponding statement \eqref{total mass 0} for the deterministic limiting profile.

\begin{remark} We expect that \eqref{Tprec} should in fact hold under the weaker conditions $\eta\gg \frac{1}{N}$ and $N \ll W^2$. The improved local semicircle law \eqref{decc} should also hold under these weaker conditions.
In particular, this would imply complete delocalization of the eigenvectors for all $N \ll W^2$. 
One  obstacle is that a non-trivial control on $\Lambda$ in the regime $\eta \le \frac{1}{W}$ is difficult to obtain.
\end{remark}

\begin{remark} \label{rem: old result}
In \cite{EK1, EK2} a diffusion approximation was proved
for $\E \absb{(\me^{-\ii tH})_{xy}}^2$  up to times $t\ll W^{1/3}$; this result was established only in a \emph{weak sense}, i.e.\ by integrating against a test function in $x-y$, living on the diffusive scale $W t^{1/2}$.
  The  formula
\begin{equation} \label{res unit}
  \frac{1}{H-E-\ii \eta} \;=\; \ii \int_0^\infty \me^{-\ii t(H-E-\ii\eta)} \, \dd t
\end{equation}
relates the resolvent with the unitary time evolution. Notice that
 the time integration can be truncated at $t\le t_0$ with $t_0$
slightly larger than $\eta^{-1}$. Hence, controlling the resolvent whose spectral parameter has imaginary part greater than $\eta$ is basically equivalent to controlling the unitary time evolution up to time $t = \eta^{-1}$. Although it was not
explicitly worked out in \cite{EK1, EK2}, the control
on $\me^{-\ii tH}$ up to $t\ll W^{1/3}$ allows one to control the
resolvent for $\eta \gg W^{-1/3}$. Theorem \ref{lm:withprof} 
(combined with Theorem \ref{lm:noprof})
is thus stronger than the results of \cite{EK1, EK2} in the following three senses.
\begin{enumerate}
\item
The resolvent is controlled for $\eta \geq W^{-1/2}$ (instead of $\eta \gg W^{-1/3}$).
\item
The control on the profile is pointwise in $x$ and $y$ (instead of in a weak sense on the scale $W \eta^{-1/2}$).
\item
The estimates hold with high probability (instead of in expectation).
\end{enumerate}
However, the result in the current paper is not uniform in $N$, unlike that of \cite{EK1, EK2}.
\end{remark}

We conclude this section with an asymptotic result on the deterministic profile $\Tdet_{x0}$. Since we are interested in large values of $x$, we need to consider the small-momentum behaviour of the Fourier transform of $\Tdet_{x0}$. Using the small-$p$ expansion \eqref{low-p exp} and \eqref{approx T hat}, we therefore find that $\Tdet_{x0} \approx \Pdet_x$, where we defined the $N$-periodic function
\begin{equation} \label{def fra P}
\Pdet_x \;\deq\; \frac{\abs{m}^2}{N}\sum_{p \in \frac{2\pi}{ N} \bZ }  \me^{\ii px} \frac{1}{\alpha \eta + W^2 D p^2} \;=\; \frac{|m|^2}{2W\sqrt{D\al\eta}} \,  \sum_{k\in \bZ}
 \exp \qbb{- \frac{\sqrt{\al\eta}}{W \sqrt{D}} \, \absb{x+kN}}\,;
\end{equation}
here the second equality follows by Poisson summation and the  Fourier transform $\int\me^{\ii px} (1 + p^2)^{-1}\, \dd p = \pi \me^{-\abs{x}}$. 
 The following proposition, proved in Appendix \ref{sec:prof}, gives the precise statement.

\begin{proposition}[Deterministic diffusion profile]\label{prop:profile}
Assume $d = 1$ and \eqref{decay of f}. For each $K \in \N$ we have
\begin{equation}
\label{prof1}
\Tdet_{xy} \;=\;
\Pdet_{x - y}
+
 O\pbb{\frac{1}{W^2}}
+O_K \pBB{\frac{1}{W} \avgbb{\frac{\sqrt{\eta}\,|x-y|_N}{W}}^{-K}}
\end{equation}
uniformly for $x$, $y$, and $z \in \f S$ with $\eta \ll 1$.

In particular,
\begin{equation} \label{supexpl}
\max_{x,y} \Tdet_{xy} \;\asymp\; \Phi^2\,.
\end{equation}
Moreover, if $(W/N)^2\le\eta\le 1$ and $N \leq W^2$, we have the sharp upper bound $\Tdet_{xy} \le C \Upsilon_{xy}$.
\end{proposition}

\begin{remark} \label{rem:prof1}
The leading-order behaviour of \eqref{prof1} is given by \eqref{def fra P}. If $\eta \ll (W/N)^2$ then $\Pdet$ is essentially a constant, i.e.\ the profile is flat.
Conversely, if $\eta \gg (W/N)^2$ then the leading term on the right-hand side of \eqref{def fra P} is given by the term $k=0$ (by periodicity of $\Pdet$ we assume that $x \in \bb T$). This is an exponentially decaying profile on the scale $\abs{x} \sim W \eta^{-1/2}$. The shape of the profile is therefore nontrivial if and only if $\eta \gg (W/N)^2$. Note that in both of the above regimes the error terms in \eqref{prof1} are negligible compared with the main term.

The total mass of the profile $\sum_{x \in \bb T} \Pdet_x$ is given by $N$ times the term $p = 0$ in the first sum of \eqref{def fra P}: 
\begin{equation} \label{total mass 1}
\sum_{x \in \bb T} \Pdet_x \;=\; \frac{\abs{m}^2}{\alpha \eta} \;=\; \frac{\im m}{\eta} \pb{1 + O(\eta)}\,,
\end{equation}
where in the last step we used the elementary identities \eqref{id m im m} and \eqref{m2} below.  In fact, the calculation \eqref{total mass 1} is a mere consistency check (to leading order) since $\sum_x \Tdet_{x0} = \frac{\im m}{\eta}$; see \eqref{total mass 3} below. We conclude that the average height of the profile is of order $(N \eta)^{-1}$. The peak of the exponential profile has height of order $(W \sqrt{\eta})^{-1}$, which dominates over the average height if and only if $\eta \gg (W/N)^2$. The regime $\eta \gg (W/N)^2$ corresponds to the regime where $\eta$ is sufficiently large that the complete delocalization has not taken place, and
the profile is mostly concentrated in the region
$|x-y| \leq W\eta^{-1/2} \ll N$.

These scenarios are best understood in a dynamical picture in which $\eta$ is decreased down from $1$.  The ensuing dynamics of $\Pdet$ corresponds to the diffusion approximation, where the
quantum problem is replaced with a random walk of step-size of order $W$.
On a configuration space consisting of $N$ sites, such a random walk
will reach an equilibrium beyond time scales $(N/W)^2$. As observed in Remark \ref{rem: old result}, $\eta^{-1}$ plays the role of time $t$, so that in this dynamical picture equilibrium is reached for $t \sim \eta^{-1} \gg (N/W)^2$. Figure \ref{figure: profile} illustrates this diffusive spreading of the profile for
different values of $\eta$.
\end{remark}

\begin{figure}[ht!]
\begin{center}
\includegraphics{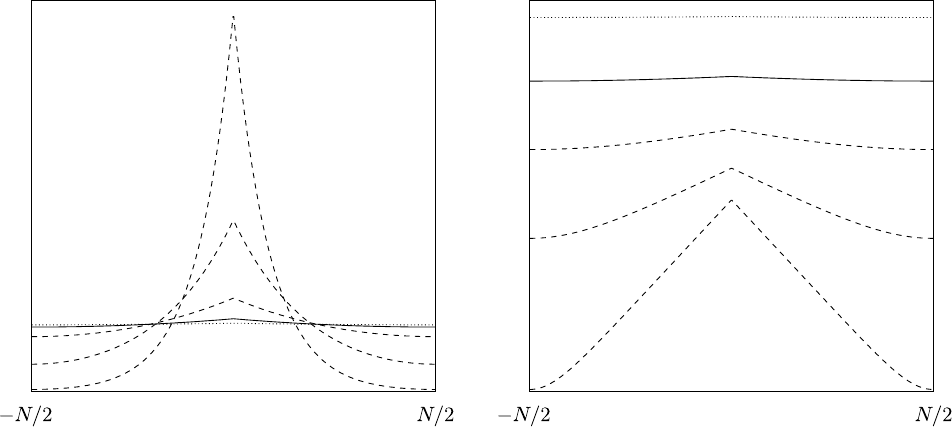}
\end{center}
\caption{A plot of the diffusion profile function at five different values of $\eta$, where the argument $x$ ranges over the torus $\bb T$. Left: the graph $x \mapsto \eta \Pdet_x$ (see \eqref{total mass 1} for the choice of normalization). Right: the graph $x \mapsto \log \Pdet_x$. Here we chose $N = 25 W$ and $\eta = 5^{-k}$ for $k = 1,2,3,4,5$. The cases $k = 1,2,3$ (where $\eta > (W/N)^2$) are drawn using dashed lines, the case $k = 4$ (where $\eta = (W/N)^2$) using solid lines, and the case $k = 5$ (where $\eta < (W/N)^2$) using dotted lines.
\label{figure: profile}}
\end{figure}

\begin{remark}
The $W$-dependent quantity $D$ in the definition \eqref{def fra P} may be replaced with the constant $D_\infty$ on the right-hand side of \eqref{def fra P}, at the expense of a multiplicative error $(1 + O(W^{-1} \log W))$ and an additive error $O(W^{-2})$. Indeed, the replacement $D \mapsto D_\infty$ in the prefactor is trivial by \eqref{Dinfty}. In order to estimate the error arising from the replacement $D \mapsto D_\infty$ in the exponent, we use the estimate $\me^{-\xi (1 + O(1/W))} = \me^{-\xi}(1 + O(x/W))$ with $\xi \sim \frac{\sqrt{\eta}}{W} \abs{x - y + kN}$. If $\xi \leq C_0 \log W$ then $x/W \leq C_0 W^{-1} \log W$, which is small enough. On the other hand, if $\xi \geq C_0 \log W$ then $\me^{-\xi} \leq W^{-C_0}$, so that the resulting error is an additive error $O(W^{-2})$.
\end{remark}

\subsection{Delocalization with a small mean-field component}
In this section we continue to assume $d = 1$ and \eqref{decay of f}.
We now consider a related model
\be\label{He}
   H_\e \;=\; (1-\e)^{1/2}H + \sqrt{\e}U\,,
\ee
where $H$ is the band matrix from Section \ref{sec:setup}, $U = (u_{ij})$ is a standard Wigner matrix independent of $H$, and 
$\e\le \frac{1}{2}$
is a small parameter. We assume that $U$ has the same symmetry type as $H$, i.e.\ either \eqref{RS} or \eqref{CH}.
Its matrix entries are normalized such that $\E u_{ij}=0$ and $\E \abs{u_{ij}}^2 = \frac{1}{N}$. Moreover, in analogy to \eqref{finite moments}, we make the technical assumption that for each $p$ there exists a constant $\mu_p$ such that $\E \abs{N^{1/2} u_{ij}}^p \leq \mu_p$ for all $N$, $i$, and $j$.

Let $S_\e = (s^{(\e)}_{ij})$ denote the
matrix of variances of the entries of $H_\e = (h_{ij}^{(\epsilon)})$, i.e. $s^{(\e)}_{ij} \deq \E \absb{h_{ij}^{(\e)}}^2$.
We find
$$
  S_\e = (1-\e) S + \epsilon \f e \f e^*\,,
$$
where we introduced the vector $\f e \deq N^{-1/2}(1,1, \dots, 1)^T$. (Hence $\f e \f e^*$ is the
matrix of the variances of $U$.)
Clearly, $0\le S_\e \le 1$ and $S_\e$ is a symmetric stochastic matrix
satisfying \eqref{stoch} and \eqref{s leq W}.

The effect of adding a small Wigner component of size $\e$ is
that the imaginary part of the spectral parameter effectively increases from $\eta$ to $\eta +\e$ in
the local semicircle law and in the diffusion approximation. 
In particular, we can eliminate the condition $N \le W^{5/4}$ and
still obtain delocalization for $H_\e$ provided $\e$ is not too small.
These results are summarized in the following theorem. In order to state it, we introduce the control parameter
\begin{equation} \label{decce}
\Phi_\e^2 \;\deq\; \max \hbb{ \frac{1}{N\eta}, \frac{1}{W\sqrt{\e+\eta}}}\,,
\end{equation}
which is analogous to $\Phi$ defined in \eqref{def Phi}.

\begin{theorem}[Delocalization with small mean-field component]\label{lm:noprof1} 
Assume $d = 1$ and \eqref{decay of f}. The following estimates hold uniformly for  $z\in \f S$.
\begin{enumerate}
\item
Suppose that $\eta (\eta + \epsilon) \gg W^{-1}$. Moreover, suppose that $N \ll W^{5/4}$ or $\eta + \epsilon \gg W^{-1/2}$. Then
\be\label{decc2}
   \Lambda^2 \;\prec\; \Phi_\epsilon^2\,.
\ee
\item
Suppose that $\epsilon + \eta \gg W^{-1/2}$ and
\begin{equation}
\frac{1}{W(\epsilon + \eta)} \;\ll\; \eta \;\leq\; \frac{W\sqrt{\epsilon + \eta}}{N}\,.
\end{equation}
Then the resolvent is completely delocalized:
\begin{equation*}
\Lambda^2 \;\prec\; \frac{1}{N\eta}\,.
\end{equation*}
\item
If $\e \gg (N/W^2)^{2/3}$ then the eigenvectors of $H_\epsilon$ are completely delocalized in the sense of Proposition~\ref{lemma: simple deloc}.
\end{enumerate}
\end{theorem}

This theorem  formulates only the bounds concerning delocalization,
i.e.\ the counterparts of Theorem~\ref{lm:noprof} and Corollary~\ref{cor:noprof}.
Similarly to Theorem~\ref{lm:withprof}, a non-trivial profile can be proved
for the average of $|G_{xy}|^2$. The profile is visible in the regime $N\eta\ge W\sqrt{\e+\eta}$,
and it is given by
\be\label{Ydete}
\Tdet^{(\e)}_{xy} \;\deq\; \pbb{\frac{\abs{m}^2 S_\e}{1 - \abs{m}^2 S_\e}}_{xy}
\;\approx\;
\frac{|m|^2 (1 - \epsilon)}{W \sqrt{D \pb{(1 - \epsilon) \alpha \eta + \epsilon}}}
\exp \qbb{ -\frac{\sqrt{(1 - \epsilon) \alpha \eta + \epsilon}}{W\sqrt{D}}|x-y|}\,,
\ee
where the approximation is valid in the regime $|x-y|\ll N$. The details of the precise formulation
and the proof are left to the reader.

\section{Preliminaries} \label{sec:prelim}
In this subsection we introduce some further notations and collect some basic facts that will be used throughout the paper. Throughout this section we work in the general $d$-dimensional setting of Section \ref{sec:setup}.

\begin{definition}[Minors] \label{def: minors}
For $ T \subset \{1, \dots, N\}$ we define $H^{( T)}$ by
\begin{equation*}
(H^{( T)})_{ij} \;\deq\; \ind{i \notin  T} \ind{j \notin  T} h_{ij}\,.
\end{equation*}
Moreover, we define the resolvent of $H^{( T)}$ through
\begin{equation*}
G^{( T)}_{ij}(z) \;\deq\;  (H^{( T)} - z)^{-1}_{ij}\,.
\end{equation*}
We also set
\begin{equation*}
\sum_i^{( T)} \;\deq\; \sum_{i \col i \notin  T}\,.
\end{equation*}
When $ T = \{a\}$, we abbreviate $(\{a\})$ by $(a)$ in the above definitions; similarly, we write $(ab)$ instead of $(\{a,b\})$.
Unless specified otherwise, summations are over the set $\{1, 2, \ldots, N\}$.
Recalling the identification of $\bb T_L^d$ with  $\{1, 2, \ldots, N\}$,
this convention is in agreement with the one given after \eqref{stoch}.
\end{definition}

\begin{definition}[Partial expectation and independence] \label{definition: P Q}
Let $X \equiv X(H)$ be a random variable. For $i \in \{1, \dots, N\}$ define the operations $P_i$ and $Q_i$ through
\begin{equation*}
P_i X \;\deq\; \E(X | H^{(i)}) \,, \qquad Q_i X \;\deq\; X - P_i X\,.
\end{equation*}
We call $P_i$ \emph{partial expectation} in the index $i$.
Moreover, we say that $X$ is \emph{independent of $T \subset \{1, \dots, N\}$} if $X = P_i X$ for all $i \in T$.
\end{definition}

The following lemma collects basic algebraic properties of stochastic domination $\prec$. We shall use it tacitly throughout Sections \ref{sec:T}--\ref{sec: ev}. Roughly it says that, under some weak restrictions, $\prec$ satisfies all the usual algebraic properties of $\leq$ on $\R$.

\begin{lemma} \label{lemma: basic properties of prec}
\begin{enumerate}
\item
Suppose that $X(u,v) \prec \Psi(u,v)$ uniformly in $u \in U$ and $v \in V$. If $\abs{V} \leq N^C$ for some constant $C$ then
\begin{equation*}
\sum_{v \in V} X(u,v) \;\prec\; \sum_{v \in V} \Psi(u,v)
\end{equation*}
uniformly in $u$.
\item
Suppose that $X_{1}(u) \prec \Psi_{1}(u)$ uniformly in $u$ and $X_{2}(u) \prec \Psi_{2}(u)$ uniformly in $u$. Then
\begin{equation*}
X_{1}(u) X_{2}(u) \;\prec\; \Psi_{1}(u) \Psi_{2}(u)
\end{equation*}
uniformly in $u$.
\item
Suppose that $\Psi(u) \geq N^{-C}$ for all $u$ and that for all $p$ there is a constant $C_p$ such that $\E \abs{X(u)}^p \leq N^{C_p}$ for all $u$. Then, provided that $X(u) \prec \Psi(u)$ uniformly in $u$, we have
\begin{equation*}
P_a X(u) \;\prec\; \Psi(u) \qquad \text{and} \qquad Q_a X(u) \;\prec\; \Psi(u)
\end{equation*}
uniformly in $u$ and $a$.
\end{enumerate}
\end{lemma}
\begin{proof}
The claims (i) and (ii) follow from a simple union bound. The claim (iii) follows from Chebyshev's inequality, using a high-moment estimate combined with Jensen's inequality for partial expectation. We omit the details.
\end{proof}

Note that if for any $\e>0$ and $p\geq 1$ we have
\be
   \E |X|^p \leq N^\e \Psi^p
\label{XP}
\ee
for large enough $N$ (depending on $\epsilon$ and $p$)
then $X\prec\Psi$ by Chebyshev's inequality. Moreover, if $X\leq \Psi$ almost surely,
then $X\prec \Psi$. Hence $O_\prec(\Psi)$ describes a larger class of random variables than $O(\Psi)$.

We need the following bound on $\Lambda$.

\begin{lemma} \label{lm:lsc}
We have
\begin{equation} \label{rough estimate on Lambda}
\Lambda(z) \;\prec\; \frac{1}{\sqrt{M \eta}}\,
\end{equation}
uniformly for $z\in \bS$.
\end{lemma}

Away from the spectral edges, i.e.\ for $z\in \bS$,  this bound was proved in 
Proposition 3.3 of \cite{EYY1}. In \cite{EYY1}, the matrix entries $x_{ij}$ were assumed to have at most 
subexponential tails (a stronger assumption than \eqref{finite moments} for all $p$), but the proof of \cite{EYY1} extends trivially to our case.
See \cite{EKYY4} for a simplified and generalized alternative proof.

The following result collects some elementary facts about $m$.
\begin{lemma} \label{lemma: msc}
We have the identity
\begin{equation} \label{id m im m}
1 - \abs{m}^2 \;=\; \frac{\eta \abs{m}^2}{\im m}\,.
\end{equation}
There is a constant $c > 0$ such that
\begin{equation} \label{m is bounded}
c \;\leq\; \abs{m} \;\leq\; 1
\end{equation}
on $z \in \f S$.
Furthermore,
\begin{equation} \label{m2}
|m|^2 \;=\; 1 - \eta\alpha + O(\eta^2)
\end{equation}
for $z \in \f S$, with $\alpha$ given in \eqref{def alpha}.
We also have the bounds
\begin{equation}\label{m3}
\im m \;\asymp\; 1 \,, \qquad 1 - |m|^2 \;\asymp\; \eta
\end{equation}
for $z \in \f S$. (The implicit constants in the two latter estimates depend on $\kappa$.)
\end{lemma}
\begin{proof}
The identity \eqref{id m im m} follows by taking the imaginary part of \eqref{identity for msc}. The estimate \eqref{m is bounded} was proved in \cite{EYY2}, Lemma 4.2. From \eqref{explicit m} we find $\im m = 1/\alpha + O(\eta)$, from which \eqref{m2} follows easily using \eqref{id m im m}. Finally, \eqref{m3} follows from Lemma 4.2 in \cite{EYY2} combined with \eqref{id m im m} and \eqref{m is bounded}.
\end{proof}

The following resolvent identities form the backbone of all of our proofs. They first appeared in \cite[Lemmas 4.1 and 4.2]{EYY1} and \cite[Lemma 6.10]{EKYY2}.
 The idea behind them is that a resolvent entry $G_{ij}$
depends strongly on the $i$-th and $j$-th columns of $H$, but weakly on all other columns. The first set of identities (called Family A) determine how to make
a resolvent entry $G_{ij}$ independent of an additional index $k \neq i,j$.
The second set (Family B) identities express the dependence of a resolvent entry $G_{ij}$ on the entries in the $i$-th or in the $j$-th column of $H$.

\begin{lemma}[Resolvent identities; {\cite[Lemma 3.5]{EKY2}}] \label{lemma: res id}
For any Hermitian matrix $H$ and $T \subset \{1, \dots, N\}$ the following identities hold.
\begin{description}
\item[Family A.] For $i,j,k \notin T$ and $k \neq i,j,$ we have
\begin{equation} \label{resolvent expansion type 1}
G_{ij}^{(T)} \;=\; G_{ij}^{(Tk)} + \frac{G_{ik}^{(T)} G_{kj}^{(T)}}{G_{kk}^{(T)}}\,, \qquad \frac{1}{G_{ii}^{(T)}} \;=\; \frac{1}{G_{ii}^{(Tk)}} - \frac{G_{ik}^{(T)} G_{ki}^{(T)}}{G_{ii}^{(T)} G_{ii}^{(Tk)} G_{kk}^{(T)}}\,.
\end{equation}
\item[Family B.]
For $i,j \notin T$ satisfying $i \neq j$ we have
\begin{equation}
\label{res exp 2b}
G_{ij}^{(T)} \;=\; - G_{ii}^{(T)} \sum_{k}^{(Ti)} h_{ik} G_{kj}^{(Ti)} \;=\; - G_{jj}^{(T)} \sum_k^{(Tj)} G_{ik}^{(Tj)} h_{k j}\,.
\end{equation}
\end{description}
\end{lemma}

\begin{definition} \label{def:adm ctprm}
The deterministic control parameter $\Psi$ is \emph{admissible} if
\begin{equation} \label{admissible Psi}
M^{-1/2} \;\leq\; \Psi^{(N)}(z) \;\leq\; M^{-\gamma/2}
\end{equation}
for all $N$ and $z \in \f S$. (Recall the parameter $\gamma$ from \eqref{def:S}.)
\end{definition}
A typical example of an admissible control parameter is
\begin{equation} \label{standard Psi}
\Psi(z) \;=\; \frac{1}{\sqrt{M \eta}}\,.
\end{equation}
If $\Psi$ is admissible then the lower bound in \eqref{admissible Psi} together with \eqref{hsmallerW} ensure that $h_{ij} \prec \Psi$.

The following lemma gives an expansion formula for the diagonal entries of $G$.

\begin{lemma}
Suppose that $\Lambda \prec \Psi$ for some admissible $\Psi$. Defining
\begin{equation*}
Z_i \;\deq\; \sum_{k,l}^{(i)}Q_i \pb{h_{ik}  G^{(i)}_{kl} h_{li}}\,,
\end{equation*}
we have
\begin{equation} \label{Giishort}
G_{ii} \;=\; m +m^2Z_i  + O_\prec\pb{\Psi^2 + M^{-1/2}} \,, \qquad Z_i \;\prec\; \Psi\,.
\end{equation}
\end{lemma}
\begin{proof}
The claim is an immediate consequence of Equations (9.1) and (9.2) in \cite{EKY2}. (Related but less explicit formulas were also obtained in \cite{EYY3}).
\end{proof}

\subsection{Averaging of fluctuations}\label{sec:Av}
In this section we collect the necessary results from \cite{EKY2}. The following proposition is a special case of the Fluctuation Averaging Theorem of \cite{EKY2}.

\begin{proposition} \label{prop: examples}
Suppose that $\Lambda \prec \Psi$ for some admissible control parameter $\Psi$. Then
\begin{equation}\label{GGnoQ}
 \sum_{a}^{(\mu \nu)} s_{\rho a} G_{\mu a} G_{a \nu} \;\prec\; \Psi(\Psi + M^{-1/4})^2 \,, \qquad
\sum_{a}^{(\mu)} s_{\rho a} G_{\mu a} G_{a \mu}^* \;\prec\; \Psi^{2}\,,
\end{equation}
and
\begin{equation}\label{GGwithQ}
\sum_{a}^{(\mu)}s_{\rho a} Q_a (G_{\mu a} G_{a \mu}) \;\prec\; \Psi^3  \,, \qquad
 \sum_{a}^{(\mu)} s_{\rho a} Q_a(G_{\mu a} G_{a \mu}^*) \;\prec\; \Psi^2\pb{\Psi + M^{-1/4}}^{2}\,,
\end{equation}
as well as
\be\label{GGG*}
  \sum_{a \neq b}^{(\mu)} s_{\rho a} s_{ab} G_{ba}G_{a\mu}G_{\mu b}^*  \;\prec\; \Psi^2\pb{\Psi + M^{-1/4}}^{2}\,,
\qquad
\sum_{a \neq b}^{(\mu)} s_{\rho a} s_{ba} G_{ba}G_{a\mu}G_{\mu b}^*  \;\prec\; \Psi^2\pb{\Psi + M^{-1/4}}^{2}\,.
\ee
\end{proposition}
\begin{proof}
All of these estimates follow immediately from Theorem 4.8, Lemma B.1, and Proposition B.2 of \cite{EKY2}, recalling that by assumption $\im m \geq c_\kappa$ by \eqref{m3}.
\end{proof}

The important quantity on the right-hand sides of 
\eqref{GGnoQ}, \eqref{GGwithQ} and \eqref{GGG*} is $\Psi$.
The additional factors $M^{-1/4}$ are a technical\footnote{This nuisance is necessary, however, and Proposition \ref{prop: examples} would be false without the factors of $M^{-1/4}$. See \cite[Remark 4.10]{EKY2}.} nuisance, but their precise form will play some role in the large-$\eta$
regime, where $M^{-1/4}$ is not negligible compared to $\Psi$.

To interpret these estimates, we note that each summand in \eqref{GGnoQ}, \eqref{GGwithQ}, and \eqref{GGG*} has a naive size given by $\Psi^k$, where $k$ is the number of off-diagonal resolvent entries in the summand. Without averaging, this naive size would be a sharp upper bound.
In the second estimate  in  \eqref{GGnoQ} the averaging does not improve the bound
since $G_{\mu a}G_{a\mu}^* =  |G_{\mu a}|^2$ is positive. In all other estimates,
the monomial on the left-hand side 
either has a nontrivial phase or its expectation is zero thanks to $Q_a$.   
Proposition \ref{prop: examples} asserts
 that in these cases the averaged quantity is  smaller than its individual summands.
Note that this averaging of fluctuations is effective even though the entries of
 $G$ may be strongly correlated. How many additional factors of $\Psi$ one gains
 depends on the structure of the left-hand side in a subtle way;
see Theorem~4.8 of \cite{EKY2} for the precise statement.
 For the  applications in this paper
the second bound in \eqref{GGwithQ} is especially important; here the averaging yields a gain of two extra factors of $\Psi$.

We  remark that all these bounds also hold if the weight functions
$s_{\rho a}$ are replaced with a more general weight function.
The precise definition is given in Definition 4.4 of \cite{EKY2}. 
All the weights used in this paper satisfy Definition 4.4 of \cite{EKY2}.

We also note that averaging in indices can be replaced by expectations.
We shall need the following special case of Theorem 4.15 of \cite{EKY2}.
\begin{proposition}\label{PGG}
Suppose that $\Lambda \prec \Psi$ for some admissible control parameter $\Psi$. Then for $a \neq \mu,\nu$
\begin{equation}\label{GGnoQP}
  P_a (G_{\mu a} G_{a \nu}) \;\prec\; \Psi(\Psi + M^{-1/4})^2\,. 
\end{equation}
\end{proposition}

\section{Self-consistent equation for $T$}\label{sec:T}

After these preparations, we now move on to the main arguments of this paper. Throughout this section we work in the general $d$-dimensional setting of Section \ref{sec:setup}. In this section we derive a self-consistent equation for $T$, given in Theorem \ref{thm:T}, whose error terms are controlled precisely using the fluctuation averaging from Proposition \ref{prop: examples}. In Section \ref{sec:inv} we solve this self-consistent equation; the result is given in Proposition \ref{YleqT}.

\begin{theorem}\label{thm:T}
Suppose that $\Lambda\prec \Psi$ for some admissible control parameter $\Psi$.
Then we have
\be\label{mainE2T}
   T_{xy} \;=\; |m|^2  \sum_{i}s_{xi}T_{iy} + |m|^2s_{xy}+
   O_\prec \pb{\Psi^4+\Psi^2 M^{-1/2}}\,.
\ee
In matrix notation,
\be\label{Tself}
   T \;=\; |m|^2 ST + |m|^2S + \cE  
\ee
where
the matrix entries of the error satisfy
\be\label{ceerror}
\cE_{xy} \;\prec\; \Psi^4+\Psi^2 M^{-1/2}\,.
\ee
\end{theorem}

The naive size of $T_{xy}$ is of order $\Psi^2$. Notice that the error term in
the self-consistent equation \eqref{Tself} is smaller by 
two orders. This improvement is essentially due to 
second estimate of \eqref{GGwithQ}.

\begin{remark} Instead of averaging in the first index of the resolvent
in the definition of $T$ \eqref{def:T}, we could have averaged in
the second, resulting in the quantity $T_{xy}' \deq \sum_{j} \abs{G_{xj}}^2 s_{jy}$. 
Then $T'$ satisfies the self-consistent equation
\begin{equation*}
T' \;=\; |m|^2 T' S + |m|^2S + \cE'\,,
\end{equation*}
where $\cE'$  also satisfies \eqref{ceerror}.
\end{remark}

Before the proof we mention that this result also gives a self-consistent equation for
the two-sided averaged quantity
$$
   Y_{xz}  \;\deq\; (TS)_{xz} \;=\; \sum_{iy} s_{xi} |G_{iy}|^2 s_{yz} \;=\; (S T')_{xz}\,.
$$
Taking the average $\sum_y s_{yz}$ of \eqref{mainE2T}, we get
the following corollary.

\begin{corollary}
Suppose that $\Lambda\prec \Psi$ for some admissible control parameter $\Psi$. Then we have
\be 
 Y = |m|^2 SY + |m|^2 S^2 + \cE \qquad \text{and} \qquad
 Y = |m|^2 YS + |m|^2 S^2 + \wt {\cE}\,,
\label{Yself}
\ee
where $\cE$ and $\wt {\cE}$ each satisfy \eqref{ceerror}.
\end{corollary}

The rest of this section is devoted to the proof of Theorem~\ref{thm:T}.
We begin by writing
\be
T_{xy} \;=\; \sum_{i} s_{xi} P_{i} |G_{iy}|^2 + \wt T_{xy}\,, \qquad
 \wt T_{xy} \;\deq\; \sum_{i} s_{xi} Q_i |G_{iy}|^2\,.
\label{fluc1T}
\ee
Then by the second formula in \eqref{GGwithQ}, we have
\be
 \wt T_{xy} =  
O_\prec(\Psi^4+\Psi^2M^{-1/2})\,.
\label{flucT}
\ee
Notice that   \eqref{GGwithQ} applies only to the summands $i\ne y$
in \eqref{fluc1T}. The estimate for the summand $i=y$ follows from
$$
   s_{xy} Q_y  |G_{yy}|^2 \;=\;  
 s_{xy} Q_y  |G_{yy}-m|^2 + 2 s_{xy}\re  \bar m Q_y (G_{yy}-m) 
 \;=\;  O_\prec(\Psi^2M^{-1/2})\,,
$$
where we used that $(G_{yy}-m) \prec \Psi$ (see \eqref{Giishort}) and that $\Psi$ 
is admissible, and in particular $\Psi M^{-1} \le \Psi^2 M^{-1/2}$.

We shall compute
$\sum_{i} s_{xi} P_{i} |G_{iy}|^2$ up to error terms of order $\Psi^4$.
We have the following result.
\begin{lemma}\label{prop:Ess}
Suppose that $\Lambda\prec \Psi$ for some admissible control parameter $\Psi$. Then
\be\label{mainE12T}
  P_{i}|G_{iy}|^2 \;=  \;
 |m|^2 P_{i} T_{iy} + |m|^2 \delta_{iy}
+ O_\prec\pb{\Psi^4+\Psi^2M^{-1/2}} + \delta_{iy} O_\prec(\Psi^2 + M^{-1/2})
\ee
and
\be\label{mainE1T}
 \sum_{i} s_{xi} P_{i}|G_{iy}|^2 \;=\;  
 |m|^2 \sum_{i} s_{xi} T_{iy} + |m|^2 s_{xy} 
+ O_\prec \pb{\Psi^4+\Psi^2M^{-1/2}}\,.
\ee
\end{lemma}

(It is possible to improve the last error term in \eqref{mainE12T}, but we shall not need this.) Before proving Lemma \ref{prop:Ess}, we show how it implies Theorem \ref{thm:T}.

\begin{proof}[Proof of Theorem \ref{thm:T}]
Equation \eqref{mainE2T} is an immediate consequence of \eqref{mainE1T}, \eqref{flucT}, and \eqref{fluc1T}. 
Hence Theorem~\ref{thm:T} follows from Lemma \ref{prop:Ess}.
\end{proof}

\begin{proof}[Proof of Lemma \ref{prop:Ess}]
Throughout the following we shall repeatedly need the simple estimate
\begin{equation} \label{simple estimate of Gii}
 \frac{1}{G_{ii}} \;=\; \frac{1}{m + O_\prec(\Psi)}
  \;=\; \frac{1}{m} \pb{1+ O_\prec(\Psi)} \;=\; O_\prec(1)\,,
\end{equation}
where the first step follows from $\Lambda \prec \Psi$, the second from the fact that $\Psi$ is admissible, 
and the last from \eqref{m is bounded}. In particular, for $k \neq i,j$, from \eqref{resolvent expansion type 1} we get the estimate
\begin{equation} \label{estimate on Gijk}
G_{ij}^{(k)} \;=\; G_{ij} - \frac{G_{ik} G_{kj}}{G_{kk}} \;=\; O_\prec (\delta_{ij} + \Psi + \Psi^2) \;=\; O_\prec(\delta_{ij} + \Psi)\,.
\end{equation}

We start the proof of Lemma \ref{prop:Ess} with the case $i\ne y$.
Using \eqref{Giishort} we get
\begin{align}
    |G_{iy}|^2 &\;=\;  |G_{ii}|^2\sum_{k,l}^{(i)}  h_{ik} G_{ky}^{(i)}
   G_{yl}^{(i)*}h_{li}
\notag \\
    &\;=\; |m|^2 \absB{  1+m Z_i +O_\prec \pb{\Psi^2 + M^{-1/2}}}^2
\sum_{k,l}^{(i)}  h_{ik} G_{ky}^{(i)}G_{yl}^{(i)*}h_{li}
\notag \\ \label{E|G|T}
   &\;=\; |m|^2 \pb{ 1 +2\re (mZ_i)}
\sum_{k,l}^{(i)}  h_{ik} G_{ky}^{(i)}G_{yl}^{(i)*}h_{li} 
+ O_\prec(\Psi^4+\Psi^2M^{-1/2})\,,
\end{align}
where in the last step we used $Z_i\prec \Psi$ and
the large deviation bound (see Lemma~\ref{lm:MZ})
\be
  \sum_{k}^{(i)}  h_{ik} G_{ky}^{(i)}  \;\prec\; \Psi\,.
\label{1sideldeT}
\ee
We may now compute the contribution of the main term in \eqref{E|G|T} to $P_i \abs{G_{iy}}^2$.
 Still assuming $i \neq y$, we find
\begin{align}
 P_i\sum_{k,l}^{(i)}  h_{ik} G_{ky}^{(i)}G_{yl}^{(i)*}h_{li}  &\;=\;
 \sum_{k}^{(i)} s_{ik} \absb{G_{ky}^{(i)}}^2
\notag \\
   &\;=\; \sum_{k}^{(i)} s_{ik} \absbb{G_{ky}-  \frac{G_{ki}G_{iy}}{G_{ii}}}^2
\notag \\
  &\;=\;  \sum_{k}^{(i)} s_{ik} \qbb{|G_{ky}|^2
  -2\Re \pbb{\ol G_{ky} 
\frac{G_{ki}G_{iy}}{G_{ii}}}} + O_\prec(\Psi^4)
\notag \\ \label{123T}
    &\;=\;  T_{iy} - 2\Re \sum_{k}^{(i)} s_{ik}\ol G_{ky} \frac{G_{ki}G_{iy}}{G_{ii}} +
 O_\prec\pb{\Psi^4 +\Psi^2 M^{-1}}\,.
\end{align}
In the third step we used \eqref{simple estimate of Gii}, and in the last step
 we added the missing term $k = i$ to obtain $T_{iy}$;
the resulting error term is $ O_\prec(\Psi^2M^{-1})$ since $i\ne y$.
Next, using \eqref{simple estimate of Gii} we get
\begin{align}\label{1eT}
 \sum_{k}^{(i)} s_{ik}\ol G_{ky}\frac{G_{ki}G_{iy}}{G_{ii}}
 &\;=\;\frac{1}{m}  \sum_{k}^{(i)} s_{ik} G_{ki}G_{iy} G_{yk}^*+
O_\prec(\Psi^4) \non\\ 
  &\;=\;  \frac{1}{m} \sum_{k}^{(i)} s_{ik} G_{ki}G_{iy} G_{yk}^{(i) *}+
O_\prec(\Psi^4)\,,
\end{align}
In the second step, using \eqref{resolvent expansion type 1} and \eqref{simple estimate of Gii}, 
we inserted an upper index $i$ as a preparation to taking the partial expectation $P_i$.
We obtain
\begin{align}\label{1233T}
P_i\sum_{k,l}^{(i)}  h_{ik} G_{ky}^{(i)}G_{yl}^{(i)*}h_{li} 
   &\;=\;  T_{iy} -2 \Re \sum_{k}^{(i)} s_{ik}\frac{1}{m} G_{ki}G_{iy} G_{yk}^{(i)*}
 + O_\prec\pb{\Psi^4 + \Psi^2 M^{-1}}.
\end{align}
Now we take the partial expectation in $i$ in \eqref{1233T}. Using that
$$
  P_i \pb{ G_{ki}G_{iy} G_{yk}^{(i)*}} \;=\;
G_{yk}^{(i)*} P_i \pb{ G_{ki}G_{iy}}
 \;\prec\; (\delta_{yk} + \Psi) \Psi \pb{\Psi + M^{-1/4}}^2 
$$
by Proposition~\ref{PGG} and \eqref{estimate on Gijk}, we find that $P_i$ applied to the second term
in \eqref{1233T} results in a quantity $O_\prec\big(  \Psi^2 (\Psi + M^{-1/4})^2 \big)$. 
Thus the contribution of main term in \eqref{E|G|T} to 
$P_{i}|G_{iy}|^2$ is
\be\label{contr1T}
  |m|^2 P_i\sum_{k,l}^{(i)}  h_{ik} G_{ky}^{(i)}G_{yl}^{(i)*}h_{li}
\;=\; |m|^2 P_{i} T_{iy} +O_\prec\big(  \Psi^4 +\Psi^2 M^{-1/2} \big).
\ee

Next, we look at the contribution of the term with a $Z$ in \eqref{E|G|T}:
\begin{align*}
&\mspace{-30mu}   |m|^2P_{i} \qbb{\pb{ mZ_i +\ol{mZ_i}}
\sum_{k,l}^{(i)}  h_{ik} G_{ky}^{(i)}G_{yl}^{(i)*}h_{li}}
\\
 &\;=\; |m|^2 P_{i} \qBB{\sum_{c,d}^{(i)} Q_i \pB{ h_{ic} \big( mG_{cd}^{(i)} + \ol{m} G_{cd}^{(i)*}\big)
h_{di} }
\sum_{k,l}^{(i)}  h_{ik} G_{ky}^{(i)}G_{yl}^{(i)*}h_{li}}
\\
  &\;=\; |m|^2\sum_{k,l}^{(i)} s_{ik}s_{il}  \pB{m G_{lk}^{(i)}  G_{ky}^{(i)}G_{yl}^{(i)*}
  + \ov{m} G_{lk}^{(i)*} G_{ky}^{(i)}G_{yl}^{(i)*}} + O_\prec( \Psi^2 M^{-1} + M^{-2})
\\
  &\;=\; |m|^22\Re \pbb{ m \sum_{k,l}^{(i)} s_{ik}s_{il}   G_{lk}^{(i)}  G_{ky}^{(i)}G_{yl}^{(i)*}} +  
O_\prec( \Psi^2 M^{-1}+M^{-2})\,.
\end{align*}
In the second step we assumed for simplicity that we are dealing with the complex Hermitian case \eqref{CH}; thus, thanks to $Q_i$, the only allowed pairing of the entries of $H$ imposes $c = l$ and $k = d$. (In the real symmetric case \eqref{RS}, there is one other term, \eqref{term for RS} below, which is estimated in the same way.) The error terms stem from the summands $c=d=l=k\ne y$ and $c=d=l=k=y$, respectively. Here we used \eqref{estimate on Gijk} as well 
as the bound $\E \abs{h_{ic}^4} \leq C M^{-1} s_{ic} \leq C M^{-2}$, an immediate consequence of \eqref{finite moments}
 and \eqref{hsmallerW}.
In the last step we switched the indices $k$ and $l$.
 
Now we remove the upper indices at the expense of an error of size $O_\prec(\Psi^4)$, and then add back the exceptional summation index $i$
as before. This gives
\begin{align}
    |m|^2P_{i} \pbb{2\re \big( mZ_i\big) \sum_{k,l}^{(i)}  h_{ik} G_{ky}^{(i)}G_{yl}^{(i)*}h_{li}}
 &\;=\;
|m|^22\Re \Big[ m \sum_{k,l} s_{ik}s_{il}   G_{lk} G_{ky}G_{yl}^{*}\Big]
 + O_\prec(\Psi^4 +\Psi^2M^{-1}  + M^{-2})
\notag \\ \label{ZET}
 &\;=\; O_\prec \pb{\Psi^4+\Psi^2M^{-1/2}}\,,
\end{align}
where in the second step we used \eqref{GGG*}; the various cases of coinciding indices $k,l,y$ are easily dealt with using the bound $M^{-1/2} \leq \Psi$.

As remarked above, in the real symmetric case \eqref{RS} the pairing $c=k$, $d=l$ is also
possible. This gives rise to the additional error term
\begin{equation} \label{term for RS}
  \sum_{k,l}  s_{ik}s_{il}   G_{kl} G_{ky}G_{yl}^{*} \;\prec\; \Psi^4+\Psi^2M^{-1/2}\,,
\end{equation}
where we used \eqref{GGG*}.

Combining \eqref{E|G|T}, \eqref{contr1T} and \eqref{ZET} yields 
\be
 P_{i}|G_{iy}|^2 \;=\;  
 |m|^2  P_{i} T_{iy}  + O_\prec \pb{\Psi^4+\Psi^2M^{-1/2}}
\label{gr3T}
\ee
for $i \neq y$.
This proves \eqref{mainE12T} for the case $i\ne y$.

If $i=y$ we compute
\be\label{egiiT}
  P_{y}|G_{yy}|^2
  =   |m|^2 + P_{y}|G_{yy}-m|^2  +2 P_y \re \big[ m  (G_{yy}-m)\big]\\
  =  |m|^2  + O_\prec\big(\Psi^2 + M^{-1/2})\,.
\ee
Here we used that  $G_{yy}-m \prec \Psi$ and that 
 $P_y (G_{yy}-m) \prec \Psi^2 + M^{-1/2}$ by \eqref{Giishort}. It is possible 
to compute this term to high order in $\Psi$, but we shall not need this.

For the proof of \eqref{mainE1T} we run almost the same argument as above
but now we 
aim at removing all upper indices $i$. We first consider the summands $i \neq y$. From  \eqref{1233T} we get
\begin{equation}\label{1234T}
P_i\sum_{k,l}^{(i)}  h_{ik} G_{ky}^{(i)}G_{yl}^{(i)*}h_{li} 
   \;=\;  T_{iy} -2 \Re \sum_{k} s_{ik}\frac{1}{m} G_{ki}G_{iy} G_{yk}^*
 + O_\prec\pb{\Psi^4 + \Psi^2 M^{-1/2}}\,,
\end{equation}
where we removed the upper index $i$ using \eqref{resolvent expansion type 1}, and included the summand $k = i$ at the expense of a negligible error term.
Taking the average $\sum_{i}^{(y)} s_{xi}$ of the second term on the right-hand side yields
\be
\sum_{i}^{(y)}\sum_ks_{xi} s_{ik}\frac{1}{m} G_{ki}G_{iy} G_{yk}^* \;=\;
\sum_{i,k}s_{xi} s_{ik}\frac{1}{m} G_{ki}G_{iy} G_{yk}^* + O_\prec(\Psi^2M^{-1}+ M^{-2})
\;=\;  O_\prec\pb{\Psi^4 + \Psi^2M^{-1/2}}\,.
\label{gr}
\ee
In the first step we just added the exceptional index $i=y$, and estimated the additional terms 
with $i = y$ using $s_{xy} s_{yk} \leq M^{-1} s_{yk} \leq M^{-2}$ as well as $G_{ky}G_{yy} G_{yk}^* \prec \delta_{ky} + \Psi^2$.
In the second step we used \eqref{GGG*}.
Note that the gain comes from the summation index $i$.

Thus the contribution of the main term of \eqref{E|G|T} to 
$\sum_{i}^{(y)} s_{xi} P_{i}|G_{iy}|^2$ is
\be\label{contr11}
 |m|^2 \sum_{i}^{(y)} s_{xi}P_i\sum_{k,l}^{(i)}  h_{ik} G_{ky}^{(i)}G_{yl}^{(i)*}h_{li} 
   \;=\; |m|^2  \sum_{i}^{(y)}s_{xi}T_{iy}
+O_\prec(\Psi^4+\Psi^2M^{-1/2})\,.
\ee
The contributions of the error terms in \eqref{E|G|T} to 
$\sum_{i\ne y} s_{xi} P_{i}|G_{iy}|^2$ are of order $O_\prec\pb{\Psi^4 + \Psi^2M^{-1/2}}$; this is true even without averaging (see \eqref{ZET}). Thus we have
\be
 \sum_{i}^{(y)} s_{xi} P_{i}|G_{iy}|^2 \;=\;  
  |m|^2  \sum_{i}^{(y)}s_{xi}T_{iy}
 +  O_\prec\pb{\Psi^4+\Psi^2M^{-1/2}}\,.
\label{gr4T}
\ee

Finally, we consider the case $i = y$. From \eqref{egiiT} we get
$$
    s_{xy} P_{y}|G_{yy}|^2
  \;=\;  |m|^2 s_{xy} + O_\prec\big(M^{-1}\Psi^2)
 \;=\;  |m|^2 s_{xy} + O_\prec(\Psi^4)\,.
$$
This formula provides the missing summands $i=y$ in \eqref{gr4T} and hence yields \eqref{mainE1T}.
\end{proof}

\section{Solving the equation for $T$}\label{sec:inv}

Throughout this section we work in the general $d$-dimensional setting of Section \ref{sec:setup}.
In this section we solve the self-consistent equation \eqref{Tself} 
to determine $T$. This involves inverting the matrix $1-|m|^2S$.
The stability of the self-consistent equation \eqref{Tself} is provided by the spectral gap of $S$. 
In the regime of complete delocalization, this gap is larger (and hence more effective) if we restrict $S$ to the subspace orthogonal to the vector $\f e = N^{-1/2}(1,1,\ldots ,1)^T$ (see the remarks after Lemma \ref{lm:invert} for more details). Therefore, we deal with the span of $\f e$ and its orthogonal complement separately. Define the rank-one projection
\begin{equation*}
\Pi \;\deq\; \f e \f e^*\,.
\end{equation*}
Thus, the entries $\Pi_{ij}$ of $\Pi$ are all equal to $1/N$, and $S\Pi=\Pi S=\Pi$ since $S$ is stochastic by \eqref{stoch}. The complementary projection is denoted by $\ol \Pi \deq 1 - \Pi$.

We perform this splitting on $T_{xy}$ only in the $x$ coordinate, regarding $y$ as fixed. Thus, we split
\begin{equation*}
T_{xy} \;=\; \ol T_{y} + (T_{xy} - \ol T_y)\,,
\end{equation*}
where we defined the averaged vector
$$
   \ol T_y \;\deq\; \frac{1}{N} \sum_{x}T_{xy} \;=\; \frac{1}{N}\sum_{i} |G_{iy}|^2 \;=\; 
 \frac{1}{N\eta}  \im  G_{yy}\,;
$$
here the last step follows easily by spectral decomposition of $G$.
We can use the local semicircle law, Lemma~\ref{lm:lsc}, to get
\be\label{barTbound}
   \ol T_y \;=\; \frac{\im m}{N\eta} \qbb{1 + O_\prec \pbb{\frac{1}{\sqrt{M\eta}}}}\,.
\ee
It is instructive to perform the same averaging with the deterministic profile $\Theta$:
\begin{equation} \label{total mass 3}
\frac{1}{N} \sum_x \Tdet_{xy} \;=\; \pbb{\Pi \, \frac{\abs{m}^2 S}{1 - \abs{m}^2 S}}_{yy} \;=\; \frac{1}{N} \frac{\abs{m^2}}{1 - \abs{m}^2} \;=\; \frac{\im m}{N \eta}\,,
\end{equation}
where in the last step we used \eqref{id m im m}.

Having dealt with the component $\Pi T$ in \eqref{barTbound}, we devote the rest of this section to the component $\ol \Pi T$. The following proposition contains the main result of this section.

\begin{proposition}\label{YleqT}
Suppose that $\Lambda\prec \Psi$ for some admissible control parameter $\Psi$. Then we have for all $y$
\be\label{kjyT}
   T_{xy} \;=\; \ol T_y + |m|^2 \pbb{ \frac{ S-\Pi}{1- |m|^2 S}}_{xy} + \wt\cE_{xy}\,,
\ee
where the error satisfies
\be\label{EER}
\max_{x,y} |\wt\cE_{xy}|   \;\prec\; \frac{1}{\eta+\pb{\frac{W}{L}}^2} \pb{\Psi^4+\Psi^2M^{-1/2}}\,.
\ee
\end{proposition}

The main tool in the proof of Proposition \ref{YleqT} is a control on the spectral gap of $S$ on the space orthogonal to $\f e$. In order to state it, we need the Euclidean matrix norm $\|A\|$ as well as the $\ell^\infty\to\ell^\infty$ norm of the matrix $A$,
$$
  \norm{A}_{\infty}  \;\deq\; \max_i \sum_{j} |A_{ij}|\,.
$$ 
The following lemma shows that $S$ has a spectral gap of order $(W/L)^2$ when restricted to the space orthogonal to $\f e$.
Its proof is postponed to the end of this section.
\begin{lemma}\label{lm:invert}
We have the bounds
\be\label{2to2}
    \normbb{  \frac{1}{1- |m|^2 S} \, \ol \Pi}
   \;\le\; \frac{C}{\eta+ \pb{\frac{W}{L}}^2}
\ee
and
\be\label{inftytoinfty}
    \normbb{ \frac{1}{1- |m|^2 S} \, \ol \Pi}_{\infty}
   \;\le\; \frac{C\log N}{\eta+ \pb{\frac{W}{L}}^2}\,.
\ee
\end{lemma}

In the regime of complete delocalization, $\eta\le (W/L)^2$, 
 the control on $(1-|m|^2S)^{-1}$ is stronger on the
space orthogonal to $\f e$. Indeed, in that regime the bound \eqref{2to2} is better than the trivial bound
\be
\normbb{ \frac{1}{1- |m|^2 S}} \;=\; \frac{1}{1-|m|^2} \;\le\; \frac{C}{\eta}
\label{bignorm}
\ee
from \eqref{m3}.

\begin{proof}[Proof of Proposition \ref{YleqT}]
Multiplying \eqref{Tself} by $\ol \Pi$ from the left yields
\begin{equation*}
\ol \Pi T \;=\; |m|^2 S \, \ol \Pi T + |m|^2 (S - \Pi) + \ol \Pi \cE  \,,
\end{equation*}
where we used that $S \Pi = \Pi S = \Pi$. Therefore
\begin{equation*}
\ol \Pi T \;=\; \abs{m}^2 \frac{S - \Pi}{1 - \abs{m}^2 S} + \wt \cE \,, \qquad \wt \cE \;\deq\; \frac{1}{1 - \abs{m}^2 S} \ol \Pi \cE\,.
\end{equation*}
Note that $(\ol \Pi T)_{xy} = T_{xy} - \ol T_y$.
Using \eqref{inftytoinfty} we therefore get \eqref{kjyT} whose error term satisfies
\begin{equation*}
\max_{x,y} \abs{\wt \cE_{xy}} \;\leq\; \normbb{ \frac{1}{1- |m|^2 S} \, \ol \Pi}_{\infty} \max_{x,y} \abs{\cE_{xy}} \;\prec\; \frac{1}{\eta + \pb{\frac{W}{L}}^2} \pb{\Psi^4+\Psi^2M^{-1/2}}\,.
\end{equation*}
This completes the proof of \eqref{kjyT} and \eqref{EER}.
\end{proof}

Next, we estimate $|G_{ij}-\delta_{ij} m|^2$ in terms of $T_{ij}$. In other words, we derive pointwise estimates on $G_{ij}$ from estimates on the \emph{averaged} quantity $T_{xy}$. This gives rise to an improved bound on $\Lambda$, which we may plug back into Proposition \ref{YleqT}. Thus we get a self-improving scheme which may be iterated.

\begin{lemma}\label{lm:contr}
Suppose that $\Lambda\prec \Psi$ with some admissible control parameter $\Psi$ and $T_{ij}\prec \Omega_{ij}^2$ for a family of admissible control parameters $\Omega_{ij}$ indexed by a pair $(i,j)$ (see Definition \ref{def:adm ctprm}). 
Then
\be\label{GT}
\absb{G_{ij}-\delta_{ij}m}^2 \;\prec\; \Omega_{ij}^2  +\Psi^4 +\delta_{ij} \sum_{k}\Omega_{ik}^2s_{ki}\,.
\ee
(Here we write $\Omega_{ij}^2 \deq (\Omega_{ij})^2$.)
\end{lemma}

\begin{proof} We fix the index $j$ throughout the proof. Let first $i \neq j$. Then \eqref{res exp 2b} gives
\begin{equation} \label{Gij iteration}
G_{ij} \;=\; G_{ii}  \sum_{k}^{(i)} h_{ik} G_{kj}^{(i)}\,.
\end{equation}
We shall use the large deviation bounds from Theorem \ref{thm: LDE} to estimate the sum. For that we shall need a bound on 
\begin{equation} \label{Gkj for Omega}
\sum_{k}^{(i)} s_{ik} |G_{kj}^{(i)}|^2  \;=\;
 \sum_{k}^{(i)} s_{ik}  \Big( |G_{kj}|^2
 + O_\prec \pb{|G_{ki}G_{ij}|^2 }\Big) 
\;=\;  T_{ij} - s_{ii} \abs{G_{ij}}^2 + O_\prec\pb{\Psi^4}  \;\prec\; \Omega_{ij}^2+ \Psi^4\,,
\end{equation}
where in the first step we used \eqref{resolvent expansion type 1} and \eqref{simple estimate of Gii}.
Since $G_{ii} \prec 1$, we get from \eqref{Gij iteration} and Theorem \ref{thm: LDE} (i) that
\begin{equation*}
\abs{G_{ij}}^2 \;\prec\; \Omega_{ij}^2 + \Psi^4\,.
\end{equation*}

To estimate $G_{ii}-m$, we use \eqref{Giishort}
to get
\be\label{diagprecT}
  |G_{ii} - m|^2 \;\leq\; C  |Z_i|^2  + O_\prec(\Psi^4 + M^{-1}) 
  \;\prec\; \sum_k \Omega_{ik}^2s_{ki} + \Psi^4 + M^{-1}\,.
\ee
Here we used
\begin{equation*}
\abs{Z_i}^2 \;\leq\; \absbb{\sum_{k}^{(i)} \pb{\abs{h_{ik}}^2 - s_{ik}} G^{(i)}_{kk}}^2 + \absbb{\sum_{k \neq l}^{(i)} h_{ik} G^{(i)}_{kl} h_{li}}^2 \;\prec\; M^{-1} + \sum_k \Omega_{ik}^2 s_{ki}\,,
\end{equation*}
where in the second step we used Theorem \ref{thm: LDE} (i) and (ii), with the bounds $G_{kk}^{(i)} \prec 1$ and
\begin{equation*}
\sum_{k \neq l}^{(i)} s_{ik} \absb{G_{kl}^{(i)}}^2 s_{li} \;\prec\; \sum_k \Omega_{ik}^2 s_{ki} + \Psi^4\,.
\end{equation*}
This last estimate follows along the lines of \eqref{Gkj for Omega}, whereby the error terms resulting from the removal of the upper indices are estimated by Cauchy-Schwarz; we omit the details. Finally $M^{-1}$ can be absorbed into $ \sum_k \Omega_{ik}^2s_{ki}$ by admissibility of $\Omega_{ij}$.
\end{proof}

We may now combine Proposition \ref{YleqT} and Lemma \ref{lm:contr} in an iterative self-improving scheme, which results in an improved bound on $\Lambda$. 

\begin{corollary}\label{cor:LO}  Suppose that $\Lambda \prec \Psi$
and $T_{ij} \prec \Omega^2$ for all $i$ and $j$, where $\Psi$ and $\Omega$ are admissible control parameters. Then
\be\label{LO}
   \Lambda^2 \;\prec\; \Omega^2\,.
\ee
\end{corollary}

\begin{proof}
We apply Lemma~\ref{lm:contr} to the constant control 
parameter $\Omega_{ij} = \Omega$ for each $i,j$. Thus, suppose that $T_{ij} \prec \Omega^2$ for all 
$i,j$, Lemma \ref{lm:contr} yields
\begin{equation*}
\Lambda^2 \;\prec\; \Psi^2 \qquad \Longrightarrow \qquad \Lambda^2 \;\prec\; \Omega^2 + \Psi^4\,.
\end{equation*}
Now we can iterate this estimate, $\Omega^2 +\Psi^4$ taking the role
of $\Psi^2$ in controlling $\Lambda^2$. Thus after one iteration we get
$$
  \Lambda^2 \;\prec\; \Omega^2 + (\Omega^2 +\Psi^4)^2 \;\prec\;  \Omega^2 +\Psi^8\,.
$$
After $k$ iterations we get $\Lambda^2 \prec \Omega^2 + \Psi^{2^k}$.
Since $\Omega$ and $\Psi$ are admissible, we have $\Psi^{2^k}\prec \Omega^2$ for $k\sim |\log \gamma|$. This completes
the proof.
\end{proof}

What remains is the proof of Lemma \ref{lm:invert}, which relies on Fourier transformation. We introduce the dual lattice of $\bb T \equiv \bb T_L^d$,
\begin{equation*}
P \;\equiv\; P_L^d \;\deq\; \frac{2 \pi}{L} \, \bb T_L^d\,.
\end{equation*}
For $p \in \R^d$ define
\begin{equation} \label{SP}
\wh S(p) \;\deq\; \sum_{x \in \bb T} \me^{-\ii p \sdot x} \, s_{x0} \;=\; \sum_{x \in \bb T} \me^{-\ii p \sdot x} \frac{1}{Z_{L,W}} f \pbb{\frac{x}{W}}\,.
\end{equation}
In particular, if $p \in P$ then $\wh S(p)$ is the discrete Fourier transform of $s_{x0}$.
Since $s_{xy}$ is translation invariant and $L$-periodic, we get for all $x,y \in \bb T$ that
\begin{equation*}
s_{xy} \;=\; s_{[x - y]_L 0} \;=\; \frac{1}{N} \sum_{p \in P} \me^{\ii p \sdot (x - y)} \, \wh S(p)\,.
\end{equation*}

\begin{proof}[Proof of Lemma \ref{lm:invert}]
First we show that for large enough $L$ the Euclidean matrix norm satisfies
\be
  \norm{S \ol \Pi} \;\le\; 1 - c_1 \pbb{\frac{W}{L}}^2
\label{Snorm}
\ee
with some positive constant $c_1$ depending on the profile $f$. Since the matrix entries $s_{ij}$ are translation invariant (see \eqref{sij}), it is sufficient
to compute its Fourier transform as defined in \eqref{SP}. Using the property $\wh{Su}(p) = \wh S(p) \wh u(p)$, the fact that $\wh \Pi(p) = \delta_{p0}$, and Plancherel's identity, we find
\begin{equation} \label{norm S Pi}
\norm{S \ol \Pi} \;=\; \max\hb{ \abs{\wh S(p)} \col p \in P \setminus \{0\}} \;\le\; 1 - c_1 \pbb{\frac{W}{L}}^2\,.
\end{equation}
The last step follows easily from $\wh S(p)\ge -1+\delta$ (recall \eqref{Slow}) and the representation
\begin{equation*}
1 - \wh S(p) \;=\; \sum_{x \in \bb T} \pb{1 - \cos(p \sdot x)} \frac{1}{Z_{L,W}} f \pbb{\frac{x}{W}} \;\geq\; c
\sum_{x \in \bb T} (p\sdot x)^2 \frac{1}{Z_{L,W}} f \pbb{\frac{x}{W}} \;\geq\; c_1 \pbb{\frac{W}{L}}^2\,,
\end{equation*}
where in the second step we used $\abs{p\sdot x} \leq \pi$, and in the last step $\abs{p} \geq 2 \pi / L$.

From \eqref{m3} we get $1-|m|^2 \ge c\eta$, which, combined with \eqref{norm S Pi}, yields
\begin{equation} \label{m2 S Pi}
\abs{m}^2 \norm{S \ol \Pi} \;\leq\; 1 - c \pbb{\frac{W}{L}}^2 - c \eta\,.
\end{equation}
Thus we get
$$
   \normbb{\frac{1}{1- |m|^2 S}\, \ol \Pi\,} \;\le\; \sum_{k = 0}^\infty \abs{m}^{2k} \norm{S \ol \Pi}^k \;\leq\;
\frac{C}{\eta + \pb{\frac{W}{L}}^2}\,.
$$   
This is \eqref{2to2}.

In order to prove \eqref{inftytoinfty}, we first observe that $\|S\|_{\infty}\le 1$ as follows from the estimate
$$
   \max_{i} |(S\bv)_{i}| \;\le\; \max_i \absbb{ \sum_{x} s_{ix} v_{x} } \;\le\; \max_{x} |v_{x}|\,,
$$
where $\f v = (v_i)$ is an arbitrary vector.
Thus, for any vector $\bv$ satisfying $\langle \bv, \f e\rangle =0$ any $k_0 \in \N$ we have
\begin{align*} 
\normbb{ \frac{1}{1- |m|^2 S} \, \bv}_\infty
  & \;\le\; \sum_{k=0}^{k_0-1} |m|^{2k} \big\| S^k \bv \big\|_\infty 
  +  \sum_{k=k_0}^\infty |m|^{2k} \big\|  S^k \bv\big\|_2
\\
&\;\le\; k_0 \|\bv \|_\infty +   \sum_{k=k_0}^\infty |m|^{2k}\| S \ol \Pi\|^{k} \|\bv\|_2 
\\
&\;\le\; k_0 \|\bv\|_\infty +   \sqrt{N}\frac{|m|^{2k_0}\| S \ol \Pi\|^{k_0}}{1-|m|^2\|S \ol \Pi\|} \|\bv \|_\infty\,,
\end{align*}
where we used the bound $\norm{\f v}_\infty \leq \norm{\bv}_2 \le \sqrt{N} \norm{\f v}_\infty$ and \eqref{m is bounded}.
Choosing $k_0 = C(\log N)[\eta +(W/L)^2]^{-1}$ with a sufficiently large constant $C$,
we obtain \eqref{inftytoinfty} exactly as above using the bound \eqref{m2 S Pi}.
This completes the proof of Lemma~\ref{lm:invert}.
\end{proof}

\section{Delocalization bounds}\label{sec:delo}
In this section we prove our main results -- Theorems \ref{lm:noprof}, \ref{lm:withprof}, and \ref{lm:noprof1}.
We return to the one-dimensional case,  $d = 1$, and continue to assume \eqref{decay of f}.
In particular, we write $N$ instead of $L$. The simple extension to higher dimensions is given in 
 Section~\ref{sec:gen}.

\subsection{Delocalization without profile: proof of Theorem \ref{lm:noprof}} \label{sec:noprof}
Suppose that $\Lambda\prec \Psi$ for some admissible control parameter $\Psi$. Then 
\eqref{kjyT} together with \eqref{EER}, \eqref{barTbound}, and \eqref{supexpl} yield
\be\label{asss}
  T_{ij} \;\prec\; \frac{1}{N\eta} + \frac{1}{W\sqrt{\eta}} +  \frac{1}{\eta+ \pb{\frac{W}{N}}^2}
\pb{\Psi^4+\Psi^2W^{-1/2}}
 \;\prec\; \Phi^2 +  \frac{1}{\eta+ \pb{\frac{W}{N}}^2}\pb{\Psi^4+\Psi^2W^{-1/2}}\,.
\ee
Recalling Corollary~\ref{cor:LO}, we have therefore proved
\be\label{stro}
\Lambda^2 \;\prec\; \Psi^2 \qquad \Longrightarrow \qquad \Lambda^2 \;\prec\; \Phi^2 +  \frac{N^2}{W^2}\pb{\Psi^4+\Psi^2W^{-1/2}}\,,
\ee
i.e. the upper bound $\Lambda^2\prec \Psi^2$ can be replaced with the stronger bound 
\eqref{stro}.

We can now iterate \eqref{stro}, exactly as in the proof of Corollary \ref{cor:LO}. We start the iteration with $\Psi_0 \deq (W \eta)^{-1/2}$; see Lemma \ref{lm:lsc}. Explicitly, the iteration reads
\begin{equation*}
\Psi_{k+1}^2 \;\deq\; \Phi^2 +  \frac{N^2}{W^2}\pb{\Psi_k^4+\Psi_k^2 W^{-1/2}}\,.
\end{equation*}
From \eqref{stro} and Lemma \ref{lm:lsc} we get that $\Lambda^2 \prec \Psi_k$ for any fixed $k$.

In order perform the iteration, we require
\be
   \frac{N^2}{W^2}\Psi_0^2 \;\ll\; 1 \qquad \text{and} \qquad  \frac{N^2}{W^2} W^{-1/2} \;\ll\; 1\,.
\label{init}
\ee
Thus we get the conditions $N \ll W^{5/4}$ and $\eta \gg N^{2} /W^3$. (Here we used \eqref{lower bound on W}). Satisfying these two conditions is the reason we need to impose the restriction on $W$ in Theorem \ref{lm:noprof}, Corollary \ref{cor:noprof}, and Theorem \ref{lm:withprof}. 
Using \eqref{init} and the fact that $\Phi$ is by definition admissible, it is now easy to see that there is a finite constant $k$, which depends on the implicit constants $c$ in $\ll$ and $\gg$ above, such that $\Psi_k^2 \leq C \Phi^2$. This concludes the proof of Theorem~\ref{lm:noprof}.

\subsection{Delocalization with profile: proof of Theorem~\ref{lm:withprof}}
By assumption we have $(W/N)^2 \leq \eta \leq 1$, so that in particular $\Phi^2 = W^{-1} \eta^{-1/2} \eqd \Psi^2$. Note that this $\Psi$ is admissible. From \eqref{decc} we get $\Lambda \prec \Psi$. Now observe that $\frac{\im m}{N \eta} = \Pi_{xy} \frac{\im m}{\eta}$ for all $x$ and $y$, as well as
\begin{equation*}
\frac{\im m}{\eta} \Pi + \abs{m}^2 \frac{S - \Pi}{1 - \abs{m^2} S} \;=\; \abs{m}^2 \frac{S}{1 - \abs{m^2} S}
\end{equation*}
by \eqref{id m im m} and the property $\Pi S = S \Pi = \Pi$.
Thus \eqref{kjyT} together with \eqref{EER} and  \eqref{barTbound} implies 
the first estimate of \eqref{Tprec},
since in the regime $\eta \ge (W/N)^2$ and $W^{5/4}\gg N$ the error term \eqref{EER} is bounded by
$$
\frac{1}{\eta}\pb{\Psi^4+\Psi^2W^{-1/2}} \;\le\; \frac{C}{N\eta}\,.
$$
The second estimate of \eqref{Tprec} follows from the first one and \eqref{mainE12T}.

Next, \eqref{Tfin} follows by using \eqref{Tdetbound} in \eqref{Tprec}.

Finally, using Lemma~\ref{lm:contr} with $\Omega_{ij}^2=\Upsilon_{ij}$ and
 $ \Psi\deq W^{-1/2}\eta^{-1/4}$,
 we obtain
\be\label{stro2}
\absb{G_{ij}-\delta_{ij}m}^2 \;\prec\; \Upsilon_{ij} +\Psi^4+ \delta_{ij}\sum_k \Upsilon_{ik} s_{ki}
 \;\prec\;  \Upsilon_{ij} .
\ee
Here we used that $\Psi^4$ can be absorbed into $ (N\eta)^{-1}\le \Upsilon_{ij}$ and in the last summation 
$\sum_k \Upsilon_{ik} s_{ki}$ can be
absorbed into $\Upsilon_{ii} \ge \frac{C}{W\sqrt{\eta}}$. This proves \eqref{Tfin1}, and hence concludes the proof of Theorem \ref{lm:withprof}.

\subsection{Delocalization with a small mean-field component: proof of Theorem~\ref{lm:noprof1}}
In order to prove Theorem \ref{lm:noprof1} we repeat the arguments from the previous sections almost to the letter. The self-consistent equations from Theorem~\ref{thm:T} remains unchanged except that $S_\e$ replaces $S$ in \eqref{Tself}.
The key observation is that, on the subspace orthogonal to $\f e$, 
the lower bound on $1-|m|^2S_\e$ is better than that on $1 - \abs{m}^2 S$. Indeed, using \eqref{m3} we get
$$
(1-|m|^2 S_\e) \ol \Pi \;=\; 1 - \abs{m}^2 (1 - \epsilon) S \ol \Pi \;\geq\; 1 - (1 - c \eta) (1 - \epsilon) S \ol \Pi
   \;\ge\; c(\eta +\e)
$$
with some positive constant $c$.
 This implies that \eqref{2to2} 
and \eqref{inftytoinfty} can be improved
to
\be\label{2to21}
    \normbb{ \frac{1}{1- |m|^2 S_\e} \, \ol \Pi}
   \;\le\; \frac{C}{\eta+\e+ \pb{\frac{W}{N}}^2}\,,
\qquad 
    \normbb{  \frac{1}{1- |m|^2 S_\e} \, \ol \Pi}_{\infty}
   \;\le\; \frac{C\log N}{\eta+\e+ \pb{\frac{W}{N}}^2}\,.
\ee

Suppose now that $\Lambda\prec\Psi$ for some admissible control parameter $\Psi$.
Then the statement of Proposition \ref{YleqT} is modified to
\be\label{kjyT1}
   T_{xy} \;=\; \ol T_y + |m|^2 \pbb{ \frac{ S_\e-\Pi}{1- |m|^2 S_\e}}_{xy} + \wt\cE_{xy}\,,
\ee
where the error term satisfies
\be\label{EER1}
\max_{x,y} |\wt\cE_{xy}|   \;\prec\; \frac{1}{\eta+\e+ \pb{\frac{W}{N}}^2} \pb{\Psi^4+\Psi^2W^{-1/2}}\,.
\ee
Notice that the Fourier transforms of $S$ and $S_\e$ (defined by \eqref{SP}) satisfy
\begin{equation} \label{S epsilon}
\wh S_\e(p)= (1-\e) \wh S(p) + \e \delta_{p0}.
\end{equation}
Thus we have
\begin{align}
 \pbb{\frac{ S_\e}{1-|m|^2 S_\e}}_{xy}
&\;=\; \frac{1}{N} \sum_{p\in P \col p\ne 0} 
\me^{\ii p(x-y)} \frac{(1-\e)\wh S(p) }{1-|m|^2(1-\e) \wh S(p)} + \frac{\im m}{\abs{m}^2 N \eta}
\notag \\ \label{Ssum1}
&\;=\; \frac{1}{N} \sum_{p\in P} 
\me^{\ii p(x-y)} \frac{(1-\e)\wh S(p) }{1-|m|^2(1-\e) \wh S(p)} + \frac{\im m}{\abs{m}^2 N \eta} + O \pbb{\frac{1}{(\eta + \epsilon) N}}\,.
\end{align}
Here we treated the zero mode $p = 0$ separately; it is given by
\begin{equation*}
\frac{1}{N} \sum_x \pbb{\frac{ S_\e}{1-|m|^2 S_\e}}_{xy} \;=\; \pbb{\Pi \frac{ S_\e}{1-|m|^2 S_\e}}_{xy} \;=\; \frac{1}{N (1 - \abs{m}^2)} \;=\; \frac{\im m}{\abs{m}^2 N \eta}\,,
\end{equation*}
where in the last step we used \eqref{id m im m}. The error term in \eqref{Ssum1} is estimated using a similar calculation.

Notice that the coefficient of $\wh S(p)$ in the denominator of \eqref{Ssum1} is now
$|m|^2(1-\e) = 1-\epsilon - (1 - \epsilon) \al \eta + O(\eta^2)$, where we used \eqref{m2}.
The results and the proof of Proposition~\ref{prop:profile} remain
unchanged when $S$ is replaced with $S_\epsilon$, except that $\alpha\eta$ must be replaced with $(1 - \epsilon) \alpha \eta+\e$ 
on the right-hand side of \eqref{def fra P}, and the whole expression is multiplied by an additional factor $(1-\e)$. 
Moreover, instead of \eqref{supexpl}, we now have
\be
  \max_{xy} \pbb{\frac{ S_\e}{1-|m|^2 S_\e}}_{xy}
  \;\asymp\; \frac{1}{N\eta} +\frac{1}{W\sqrt{\eta+\e}}\,.
\label{supexpl1}
\ee

Recall the definition \eqref{decc} of $\Phi_\epsilon$.
Following the proof of Theorem~\ref{lm:noprof}, instead of \eqref{stro} we now obtain
\be\label{stro1}
\Lambda^2 \;\prec\; \Psi^2 \qquad \Longrightarrow \qquad \Lambda^2  \;\prec\; \Phi_\e^2 +  \frac{1}{\eta+ \e+ \pb{\frac{W}{N}}^2} \pb{\Psi^4+\Psi^2W^{-1/2}}\,.
\ee
As in Section \ref{sec:noprof}, we can iterate \eqref{stro1} under the conditions
\begin{equation} \label{assump for iter eps}
\frac{1}{W \eta} \;\ll\; \eta + \epsilon + \pbb{\frac{W}{N}}^2 \,, \qquad W^{-1/2} \;\ll\; \eta + \epsilon + \pbb{\frac{W}{N}}^2\,.
\end{equation}
(Note that the a priori estimate  $(W \eta)^{-1}$ is still determined by $W$ despite the
small mean-field component. In Lemma~\ref{lm:lsc} it is given by $(M\eta)^{-1/2}$ where 
$M = (\max_{ij} s_{ij})^{-1} \sim (\e N^{-1} + W^{-1})^{-1}\sim W$.) The first condition of \eqref{assump for iter eps} holds if
\begin{equation} \label{cond1}
\eta (\eta + \epsilon) \;\gg\; W^{-1}\,,
\end{equation}
and the second holds if either 
\begin{equation} \label{cond2}
\eta + \epsilon \;\gg\; W^{-1/2 }
\end{equation}
or
\begin{equation}
N \;\ll\; W^{5/4 }\,.
\end{equation}
This concludes the proof of part (i).

In order to get complete delocalization of the resolvent, i.e.\ $\Lambda^2 \prec (N \eta)^{-1}$, we require $\Lambda \prec \Phi_\epsilon^2$ as well as
\begin{equation} \label{cond3}
W \sqrt{\eta + \epsilon} \;\geq\; N \eta\,,
\end{equation}
which ensures that $\Phi_\epsilon = (N \eta)^{-1}$.
Hence we get complete delocalization of the resolvent provided that \eqref{cond1}, \eqref{cond2}, and \eqref{cond3} hold. This concludes the proof of part (ii).

If $\epsilon \gg  (N/W^2)^{2/3}$ then there exists an $\eta$ such that the assumptions of part (ii) are met. Hence part (ii) and Proposition \ref{lemma: simple deloc} yields part (iii). This concludes the proof of Theorem \ref{lm:noprof1}.

\section{Complete delocalization of eigenvectors} \label{sec: ev}

In this section we derive a delocalization result
for the eigenvectors of $H$, using the complete delocalization $\Lambda^2 \prec (N \eta)^{-1}$ as input. We denote the eigenvalues of $H$ by $\lambda_1 \leq \lambda_2 \leq \cdots \leq \lambda_N$, and the associated normalized eigenvectors by $\f u_1, \f u_2, \dots, \f u_N$. We use the notation $\f u_\alpha = (u_\alpha(x))_{x = 1}^N$. We shall only consider eigenvectors associated with eigenvalues lying in the interval
\begin{equation*}
I \;\deq\; [-2 + \kappa, 2 - \kappa]\,,
\end{equation*}
where $\kappa > 0$ is fixed.

For $\ell \equiv \ell(L)$ define the characteristic function $P_{x,\ell}$ projecting onto the complement of the $\ell$-neighbourhood of $x$,
\begin{equation*}
P_{x,\ell}(y) \;\deq\; \ind{\abs{y - x} \geq \ell}\,.
\end{equation*}
Let $\epsilon>0$ and define the random subset of eigenvector indices through
\begin{equation*}
\Adet_{\epsilon, \ell} \;\deq\; \hbb{\alpha \col \lambda_\alpha \in I \,,\, \sum_x \abs{u_\alpha(x)} \, \norm{P_{x,\ell} \, \f u_\alpha} \leq \epsilon}\,,
\end{equation*}
which indexes the set of eigenvectors localized on scale $\ell$ up to an error $\epsilon$; see Remark \ref{rem: deloc vect} below for more details on its interpretation.

\begin{proposition}[Complete delocalization of eigenvectors] \label{lemma: simple deloc}
Suppose that $\Lambda \prec \Psi$ for some admissible control parameter $\Psi$. Let $\eta \equiv \eta_N$ be a sequence satisfying $M^{-1 + \gamma} \leq \eta \ll 1$. Suppose that
\begin{equation} \label{large eta estimate}
\sup_{E \in I} \abs{G_{xy}(E + \ii \eta)}^2 \;\prec\; \frac{1}{N \eta} + \delta_{xy}\,.
\end{equation}
Let $\ell \ll N$. Then we have for any $\epsilon > 0$
\begin{equation*}
\frac{\abs{\Adet_{\epsilon, \ell}}}{N} \;\leq\; C \sqrt{\epsilon}+ O_\prec(N^{-c})\,.
\end{equation*}
\end{proposition}

\begin{remark} \label{rem: deloc vect}
The set $\Adet_{\epsilon, \ell}$ contains, in particular, all indices associated with eigenvectors that are exponentially localized in 
balls of radius $O(\ell)$. In fact, exactly as in \cite[Corollary 3.4]{EK1}, Proposition \ref{lemma: simple deloc} implies that the fraction of eigenvectors subexponentially localized on scales $\ell$ vanishes with high probability for large $N$.
\end{remark}

\begin{proof}[Proof of Proposition \ref{lemma: simple deloc}]
As usual, we omit the spectral parameter $z = E + \ii \eta$, where $E \in I$ is arbitrary and $\eta$ is the parameter given in the statement of Proposition \ref{lemma: simple deloc}.
By assumption on $\Lambda$, we have for all $x$
\begin{equation}\label{Gxx}
\frac{\eta}{\im m} \sum_{y} \abs{G_{yx}}^2 \; = \; \frac{\im G_{xx}}{\im m}
\;=\; 1 + O_\prec(\Psi)\,,
\end{equation}
uniformly in $E \in I$, where in the first step we used the spectral decomposition of $G$.
Thus, for all $x$, the map $y \mapsto \frac{\eta}{\im m} \abs{G_{yx}}^2$ is approximately a probability distribution on $\{1, \dots, N\}$. Roughly, \eqref{large eta estimate} states that this probability distribution is supported on the order of $N$ sites of $\{1, \dots, N\}$. More precisely, \eqref{large eta estimate} yields (introducing the standard basis  vector $\delta_x$ defined by $(\delta_x)(y) \deq \delta_{xy}$), for any fixed $x$,
\begin{align}
\frac{\eta}{\im m} \,  \normb{P_{x,\ell} \, G \, \delta_x}^2 &\;=\; \frac{\eta}{\im m} \sum_{y} \ind{\abs{y - x} \geq \ell} \,  \abs{G_{yx}}^2 
\notag \\
&\;=\;
\frac{\eta}{\im m} \sum_{y}  \abs{G_{yx}}^2 - \frac{\eta}{\im m} \sum_{y} \ind{\abs{y - x} < \ell} \,  \abs{G_{yx}}^2
\notag \\
&\;=\; 1 + O_\prec(\Psi) + O_\prec \qbb{\frac{\eta}{\im m} \pbb{\frac{N^{1 - c}}{N \eta} + 1}}
\notag \\ \label{spread of G}
&\;=\; 1 + O_\prec(N^{-c})\,,
\end{align}
uniformly in $E \in I$. Here in the third step we used \eqref{Gxx} and \eqref{large eta estimate}, and in the last step the upper bound $\eta \leq M^{-c}$ and the fact that $\Psi$ is admissible.

In order to obtain a statement about the eigenvectors, we do a spectral decomposition $G = \sum_\alpha \frac{\f u_\alpha \f u_\alpha^*}{\lambda_\alpha - z}$, which yields for arbitrary $\zeta > 0$
\begin{multline} \label{main splitting for deloc}
\frac{\eta}{\im m} \normb{P_{x,\ell} \, G \, \delta_x}^2 \;=\; \frac{\eta}{\im m} \normBB{\sum_\alpha \frac{1}{\lambda_\alpha - z} \, \ol u_\alpha(x) \, P_{x,\ell} \f u_\alpha}^2
\\
\leq\; \frac{\eta}{\im m} \pbb{1 + \frac{1}{\zeta}} \normBB{\sum_{\alpha \in \Adet_{\epsilon, \ell}} \frac{1}{\lambda_\alpha - z} \, \ol u_\alpha(x) \, P_{x,\ell} \f u_\alpha}^2 + \frac{\eta}{\im m} (1 + \zeta) \normBB{\sum_{\alpha \in \Adet_{\epsilon, \ell}^c} \frac{1}{\lambda_\alpha - z} \, \ol u_\alpha(x) \, P_{x,\ell} \f u_\alpha}^2\,,
\end{multline}
where  we introduced the complement set $\Adet_{\epsilon,\ell}^c \deq \{1, \dots, N\} \setminus \Adet_{\epsilon, \ell}$.
In order to estimate the first term on the right-hand side of \eqref{main splitting for deloc}, we write
\begin{align*}
\frac{\eta}{\im m} \normBB{\sum_{\alpha \in \Adet_{\epsilon, \ell}} \frac{1}{\lambda_\alpha - z} \, \ol u_\alpha(x) \, P_{x,\ell} \f u_\alpha}^2
&\;\leq\;
\frac{\eta}{\im m} \normBB{\sum_{\alpha \in \Adet_{\epsilon, \ell}} \frac{1}{\lambda_\alpha - z} \, \ol u_\alpha(x) \, \f u_\alpha}^2
\\
&\;\leq\; \frac{\eta}{\im m} \sum_\alpha \frac{\abs{u_\alpha(x)}^2}{\abs{\lambda_\alpha - z}^2}
\\
&\;=\; \frac{\im G_{xx}}{\im m}\;=\; 1 + O_\prec(\Psi)\,.
\end{align*}
Therefore we may estimate the left-hand side by its square root 
to get the bound
\begin{align}
\frac{\eta}{\im m} \normBB{\sum_{\alpha \in \Adet_{\epsilon, \ell}} \frac{1}{\lambda_\alpha - z} \, \ol u_\alpha(x) \, P_{x,\ell} \f u_\alpha}^2
&\;\leq\; (1 + O_\prec(\Psi)) \sqrt{\frac{\eta}{\im m}} \, \normBB{\sum_{\alpha \in \Adet_{\epsilon, \ell}} \frac{1}{\lambda_\alpha - z} \, \ol u_\alpha(x) \, P_{x,\ell} \f u_\alpha}
\notag \\ \label{first term for deloc}
&\;\leq\; \pb{C + O_\prec(\Psi)} \sum_{\alpha \in \Adet_{\epsilon,\ell}} \sqrt{\frac{\eta}{(E - \lambda_\alpha)^2 + \eta^2}} \abs{u_\alpha(x)} \, \norm{P_{x,\ell} \f u_\alpha}\,,
\end{align}
where in the last step we used \eqref{m3}.

Similarly, we may estimate the second term of \eqref{main splitting for deloc} using
\begin{equation} \label{second term for deloc}
\normBB{\sum_{\alpha \in \Adet_{\epsilon, \ell}^c} \frac{1}{\lambda_\alpha - z} \, \ol u_\alpha(x) \, P_{x,\ell} \f u_\alpha}^2 \;\leq\; \normBB{\sum_{\alpha \in \Adet_{\epsilon, \ell}^c} \frac{1}{\lambda_\alpha - z} \, \ol u_\alpha(x) \, \f u_\alpha}^2 \;=\; \sum_{\alpha \in \Adet_{\epsilon, \ell}^c} \frac{\abs{u_\alpha(x)}^2}{\abs{\lambda_\alpha - z}^2}\,.
\end{equation}

Combining \eqref{main splitting for deloc} with \eqref{spread of G}, \eqref{first term for deloc}, and \eqref{second term for deloc}, we get
\begin{multline*}
1 + O_\prec(N^{-c}) \;\leq\;\pb{C + O_\prec(N^{-c})} \pbb{1 + \frac{1}{\zeta}}\sum_{\alpha \in \Adet_{\epsilon,\ell}} \sqrt{\frac{\eta}{(E - \lambda_\alpha)^2 + \eta^2}} \, \abs{u_\alpha(x)} \, \norm{P_{x,\ell} \f u_\alpha}
\\
+ (1 + \zeta) \sum_{\alpha \in \Adet_{\epsilon, \ell}^c} \frac{\eta \, \abs{u_\alpha(x)}^2}{\im m \, \abs{\lambda_\alpha - z}^2}\,,
\end{multline*}
uniformly for $E \in I$. Now taking the average $\abs{I}^{-1} \int_I \dd E$ and using Jensen's inequality, we find
\begin{multline*}
1 + O_\prec(N^{-c}) \;\leq\; \pb{C + O_\prec(N^{-c})} \pbb{1 + \frac{1}{\zeta}} \sum_{\alpha \in \Adet_{\epsilon,\ell}} \abs{u_\alpha(x)} \, \norm{P_{x,\ell} \f u_\alpha}
\\
+ (1 + \zeta) \frac{1}{\abs{I}} \int_I \dd E \, \sum_{\alpha \in \Adet_{\epsilon, \ell}^c} \frac{\eta \, \abs{u_\alpha(x)}^2}{\im m \, \abs{\lambda_\alpha - E - \ii \eta}^2}\,,
\end{multline*}
where we used that
\begin{equation} \label{approximate delta}
\int_I \dd E \, \frac{\eta}{(E - \lambda_\alpha)^2 + \eta^2} \;\geq\; \frac{\pi}{2}
\end{equation}
for $\alpha \in \Adet_{\epsilon,\ell}$.
Averaging over $x$, i.e.\ taking $N^{-1} \sum_x$, yields
\begin{equation} \label{averaged estimate}
1 + O_\prec(N^{-c}) \;\leq\; \pb{C + O_\prec(N^{-c})} \pbb{1 + \frac{1}{\zeta}} \epsilon
+ (1 + \zeta) \frac{1}{\abs{I}} \int \dd E  \, \frac{1}{N} \sum_{\alpha \in \Adet_{\epsilon, \ell}^c} \frac{\eta}{\im m \, \abs{\lambda_\alpha - E - \ii \eta}^2}\,,
\end{equation}
where we used the definition of $\Adet_{\epsilon, \ell}$.
We may estimate the integral as
\begin{equation*}
\frac{1}{\abs{I}} \int \dd E  \, \frac{1}{N} \sum_{\alpha \in \Adet_{\epsilon, \ell}^c} \frac{\eta}{\im m \, \abs{\lambda_\alpha - E - \ii \eta}^2} \;\leq\;
\frac{1}{\abs{I}} \int \dd E  \, \frac{1}{N} \frac{\im \tr G}{\im m} \;=\; 1 + O_\prec(\Psi)\,.
\end{equation*}
Setting $\zeta = \sqrt{\epsilon}$ in \eqref{averaged estimate} therefore yields
\begin{equation*}
\frac{1}{\abs{I}} \int \dd E  \, \frac{1}{N} \sum_{\alpha \in \Adet_{\epsilon, \ell}^c} \frac{\eta}{\im m \, \abs{\lambda_\alpha - E - \ii \eta}^2} \;\geq\; 1 - \pb{C + O_\prec(N^{-c})} \sqrt{\epsilon} - O_\prec(N^{-c}) \;\geq\; 1 - C \sqrt{\epsilon} - O_\prec(N^{-c})\,.
\end{equation*}
Using
\begin{equation*}
\frac{1}{\abs{I}} \int \dd E  \, \frac{1}{N} \sum_{\alpha} \frac{\eta}{\im m \, \abs{\lambda_\alpha - E - \ii \eta}^2} \;=\;
\frac{1}{\abs{I}} \int \dd E  \, \frac{1}{N} \frac{\im \tr G}{\im m} \;=\; 1 + O_\prec(\Psi)
\end{equation*}
we therefore get
\begin{align*}
\frac{1}{\abs{I}} \int \dd E  \, \frac{1}{N} \sum_{\alpha \in \Adet_{\epsilon, \ell}} \frac{\eta}{\im m \, \abs{\lambda_\alpha - E - \ii \eta}^2} \;\leq\; C \sqrt{\epsilon} + O_\prec(N^{-c})\,.
\end{align*}
The claim now follows from \eqref{approximate delta}.
\end{proof}

\section{Extension to higher dimensions and a slowly decaying band}\label{sec:gen}

In this Section we extend Theorem \ref{lm:noprof},
Corollary \ref{cor:noprof}, and Theorem \ref{lm:withprof} in two directions: higher dimensions $d$ and a slowly decaying band.

The multidimensional analogues of the slowly decaying
profile are left to the reader, as is the formulation of
these extensions if a small mean-field component is added to the band matrix.
All these results can be obtained in a straightforward manner following the proofs for the one-dimensional case with a rapidly decaying $f$.

\subsection{Higher dimensions}
Fix $d =1,2,3, \dots$, and recall that $N = L^d$ and $M \asymp W^d$. Throughout this section we continue to assume \eqref{decay of f}.
The following lemma gives the sharp upper bound on the size of $\Tdet_{xy}$ defined by the formula \eqref{Ydet}, where $x$ and $y$ take on values in the $d$-dimensional torus $\bb T_L^d \equiv \bb T$.
\begin{lemma}\label{lm:prof bound d}
Let $d = 2,3, \dots$ and assume \eqref{decay of f}. Then there is a constant $C$ such that
\begin{equation*}
\Tdet_{xy} \;\leq\; C \max \hbb{\frac{1}{M} , \frac{1}{N \eta}}
\end{equation*}
for all $x$ and $y$.
\end{lemma}
\begin{proof}
See Appendix \ref{sec:prof}.
\end{proof}
In order to state the precise form of the profile $\Tdet_{xy}$, we define the covariance matrix $D \equiv D_W$ through
\begin{equation} \label{Dd}
D_{ij} \;\deq\; \frac{1}{2} \sum_{x \in \bb T} \frac{x_i x_j}{W^2} s_{x0}\,.
\end{equation}
We have
\begin{equation*}
D \;=\; D_\infty + O(W^{-1}) \qquad \text{where} \qquad (D_{\infty})_{ij} \;\deq\; \frac{1}{2} \int_{\R^d} x_i x_j f(x) \, \dd x\,.
\end{equation*}
Since $D_\infty > 0$ we get $D \geq c > 0$ uniformly in $W$.

Next, we define the $d$-dimensional Yukawa potential
\begin{equation*}
V(x) \;\deq\; \int_{\bR^d} \frac{\dd q}{(2 \pi)^d} \frac{\me^{\ii q \sdot x}}{1 + q^2}\qquad (x\in \bR^d)\,,
\end{equation*}
where the integral is to be understood as the Fourier transform of a tempered distribution.
For $d = 1$ we have $V(x) = \frac{1}{2} \me^{-\abs{x}}$ and for $d = 3$ we have $V(x) = \frac{1}{4 \pi \abs{x}}\me^{-\abs{x}}$. Generally, in $d \geq 3$ dimensions $V$ has a singularity of type $\abs{x}^{2 - d}$ at the origin, and in $d = 2$ dimensions a singularity of type $\log \abs{x}$; see e.g.\ \cite[Theorem 6.23]{LiebLoss} for more details.
The leading-order behaviour of the profile $\Tdet_{x0}$ is given by
\begin{equation} \label{approx prof d}
\Pdet_x \;\deq\; \frac{\abs{m}^2}{N} \sum_{p \in (\frac{2 \pi}{L} \Z)^d} \me^{\ii p \sdot x} \frac{\chi(Wp)}{\alpha \eta + W^2 (p \sdot D p)} \;=\; \frac{\abs{m}^2 (\alpha \eta)^{d/2 - 1}}{W^d \sqrt{\det D}} \sum_{k \in \Z^d}  \pb{V * \varphi_{\sqrt{\alpha \eta}}} \pbb{\frac{\sqrt{\alpha \eta}}{W}D^{-1/2} (x + kL)}\,,
\end{equation}
where $\chi$ is a smooth function satisfying $\chi(q) = 1$ for $\abs{q} \leq 1/2$ and $\chi(q) = 0$ for $\abs{q} \geq 1$, $\varphi$ is a Schwartz function satisfying $\int \varphi = 1$, and $\varphi_t(x) \deq t^{-d} \varphi(x/t)$. (In fact, $\varphi(D^{-1/2} x)$ is the Fourier transform of $\chi(q)$.) The second step of \eqref{approx prof d} follows by Poisson summation; see Appendix \ref{sec:prof} and in particular \eqref{1st} for more details. The following lemma gives the precise error bounds in the approximation \eqref{approx prof d}.

\begin{lemma}\label{lm: prof for d}
Let $d = 1,2,3,\dots$ and assume \eqref{decay of f}. Then
\begin{equation} \label{asymp for T d}
\Tdet_{xy} \;=\; \Pdet_{x - y}
+ O_K \pBB{\frac{1}{W^d} \avgbb{\frac{x}{W}}^{-K} + \frac{\eta^{d/2}}{W^d} \avgbb{\frac{\sqrt{\eta} x}{W}}^{-K}}\,.
\end{equation}
\end{lemma}
\begin{proof}
See Appendix \ref{sec:prof}.
\end{proof}

\begin{remark}
The convolution in \eqref{approx prof d} smooths out the Yukawa potential on the scale $x \approx W$. The error terms in \eqref{asymp for T d} are negligible compared to the main term $\Pdet_x$ in the regime $W \ll |x| \leq C W \eta^{-1/2}$. Therefore the approximation $\Pdet$ is meaningful from the profile scale $W \eta^{-1/2}$ down to the band scale $W$. The actual choice of the function $\chi$
in \eqref{approx prof d} is immaterial in the relevant regime $|x|\gg W$,
as long as $\chi$ is equal to one in a neighbourhood of the origin.
\end{remark}

Next, we state the counterparts of Theorem \ref{lm:noprof}, Corollary \ref{cor:noprof}, and Theorem \ref{lm:withprof} in the higher-dimensional setting.  Their proofs are trivial modifications of the proofs of their one-dimensional counterparts, using Lemmas \ref{lm:prof bound d} and \ref{lm: prof for d}.

\begin{theorem}[Improved local semicircle law]\label{lm:noprofd}
Let $d = 2,3, \dots$ and assume \eqref{decay of f}. Suppose moreover that $L \ll W^{1 + d/4}$ and $\eta \gg L^{2} / W^{d + 2}$. Then we have
\be
   \Lambda^2 \;\prec\; \max \hbb{\frac{1}{M} , \frac{1}{N \eta}}
\ee
for $z\in \f S$.
\end{theorem}

\begin{corollary}[Eigenvector delocalization]\label{cor:noprofd}
Let $d = 2,3, \dots$ and assume \eqref{decay of f}. If  $L \ll W^{1 + d/4}$ then the eigenvectors of $H$ are completely delocalized in the sense of Proposition \ref{lemma: simple deloc}.
\end{corollary}

\begin{theorem}[Diffusion profile]\label{lm:withprofd}
Let $d = 2,3, \dots$ and assume \eqref{decay of f}. Suppose that $L \ll W^{1 + d/4}$ and $(W/L)^2 \le \eta \le 1$.
Then
\be
 |T_{xy} -\Tdet_{xy}| \;\prec\; \frac{1}{N\eta} + \frac{1}{M^{3/2} \eta}\,, \qquad
  \absB{P_x |G_{xy}|^2 - \delta_{xy}|m|^2- \abs{m}^2 \Tdet_{xy}} \;\prec\; \frac{1}{N\eta} + \frac{1}{M^{3/2} \eta}
 + \frac{\delta_{xy}}{\sqrt{M}}\,.
\ee
Moreover, the analogues of \eqref{Tfin} and \eqref{Tfin1} hold with
\begin{equation*}
\Upsilon_{xy} \;\equiv\; \Upsilon^{(K)}_{xy} \;\deq\; \frac{\eta^{d/2 - 1}}{W^d} \pb{V * \varphi_{\sqrt{\eta}}} \pbb{\frac{\sqrt{\alpha \eta}}{W}D^{-1/2} (x - y)} + \frac{1}{W^d} \avgbb{\frac{x-y}{W}}^{-K} + \frac{\eta^{d/2}}{W^d} \avgbb{\frac{\sqrt{\eta} (x-y)}{W}}^{-K}\,,
\end{equation*}
where $K$ is an arbitrary, fixed, positive integer.
\end{theorem}

\subsection{Slowly decaying band}
In this section we make the following assumption on the band shape.
Suppose that $d = 1$ and $f$ is smooth and symmetric, and satisfies
\begin{equation} \label{heavy tail band}
f(x) \;=\; \frac{h(x)}{\abs{x}^{1 + \beta}}
\end{equation}
for some fixed $\beta \in (0,2)$. Here $h$ is a
 symmetric function satisfying
\begin{equation} \label{assumption on h}
\abs{h(x) - h_0} \;\leq\; C \avg{x}^{-3}\,
\end{equation}
for some fixed $h_0>0$.
Note that by definition $f$ is
 smooth and symmetric, so that $h(x) = O(\abs{x}^{1 + \beta})$ near the origin.

In order to avoid technical issues arising from the periodicity of $S$, we cut off the tail of $f$ at scales $x \approx N$. Thus we set
\begin{equation*}
s_{xy} \;\deq\; \frac{1}{Z} f \pbb{\frac{[x - y]_N}{W}} \sigma\pbb{\frac{[x - y]_N}{N}}\,;
\end{equation*}
here $\sigma$ is a smooth, symmetric bump function satisfying $\sigma(x) = 1$ for $\abs{x} \leq a$ and $\sigma(x) = 0$ for $\abs{x} \geq b$, where $0 < a < b < 1/2$. As usual, $Z$ is a normalization constant.

The following lemma is the analogue of Lemma \ref{lm:invert}. Its proof is similar to that of Lemma \ref{lm:invert}; the key input is Lemma \ref{lm:SW heavy tail} (iii).
\begin{lemma} \label{lm: ht gap}
Suppose that $d = 1$ and that \eqref{heavy tail band} and \eqref{assumption on h} hold. Then
\be\label{inftytoinftyht}
    \normbb{ \frac{1}{1- |m|^2 S} \, \ol \Pi}_{\infty}
   \;\leq\; \frac{C\log N}{\eta+ \pb{\frac{W}{N}}^{\beta}}\,.
\ee
\end{lemma}

Next, we give the sharp upper bound on the peak of the profile.
\begin{lemma} \label{lm:ht profile bound}
Suppose that $d = 1$ and that \eqref{heavy tail band} and \eqref{assumption on h} hold. Then
\begin{equation*}
\Tdet_{xy} \;\leq\; C \max \hbb{\frac{\eta^{1/\beta - 1} + 1}{W}, \frac{1}{N \eta}}\,.
\end{equation*}
\end{lemma}
\begin{proof}
See Appendix \ref{sec:prof}.
\end{proof}

In order to describe the asymptotic shape of the profile, we define
\begin{equation*}
B \;\deq\; h_0 \frac{W}{Z} \int_\bR \dd u \, \frac{1 - \cos u}{\abs{u}^{1 + \beta}} \;=\;
 h_0 \int_\bR \dd u \, \frac{1 - \cos u}{\abs{u}^{1 + \beta}} +
 O \pBB{\frac{1}{W} + \pbb{\frac{W}{N}}^\beta}\,,
\end{equation*}
which plays a role similar to the unrenormalized diffusion constant $D$ from \eqref{def D}. Moreover, define the function
\begin{equation*}
V(x) \;\deq\; \int_\bR \frac{\dd q}{2 \pi} \frac{\me^{\ii qx}}{1 + \abs{q}^\beta}\,,
\end{equation*}
which is bounded for $\beta > 1$. It is easy to check that for $\beta > 1$ and $\abs{x} \geq 1$ we have
\begin{equation} \label{asymp for V}
V(x) \;=\;
C_\beta \abs{x}^{-1-\beta} + O(\abs{x}^{-2-\beta})
\end{equation}
with an explicitly computable constant $C_\beta>0$.

\begin{proposition}\label{lm:profile heavy tail}
Suppose that $d = 1$ and that \eqref{heavy tail band} and \eqref{assumption on h} hold for some $\beta > 1$. Suppose moreover that
\begin{equation} \label{prof assump ht}
\pbb{\frac{W}{N}}^\beta \;\ll\; \eta \;\ll\; 1\,.
\end{equation}
Then there is a constant $c > 0$, depending on the implicit exponents in \eqref{prof assump ht}, such that for $x \in \bb T$ we have
\begin{equation} \label{Tdet ht}
\Tdet_{x0} \;=\; \frac{|m|^2}{W \alpha \eta} \pbb{\frac{\alpha \eta}{B}}^{1/\beta} \,
 V \qbb{\pbb{\frac{\alpha \eta}{B}}^{1/\beta} \, \frac{x}{W}} + O \pbb{\frac{\eta^{1/\beta - 1}}{W^{1 + c}}}
\end{equation}
\end{proposition}
\begin{proof}
See Appendix \ref{sec:prof}.
\end{proof}

The matrix $\Tdet$ is the resolvent of a \emph{superdiffusive} operator, whose symbol in Fourier space is $B \abs{W p}^\beta$. Thus, under the identification $t = \eta^{-1}$ from Remark \ref{rem: old result}, we find that the associated dynamics scales according to $x \sim W t^{1/\beta}$ instead of the diffusive scaling $x \sim W t^{1/2}$. 

We may now state the counterparts of Theorem \ref{lm:noprof}, Corollary \ref{cor:noprof}, and Theorem \ref{lm:withprof} for the slowly decaying band. Their proofs are trivial modifications of those for the strongly decaying band, using Lemmas \ref{lm: ht gap} and \ref{lm:ht profile bound}.

\begin{theorem}[Improved local semicircle law]\label{lm:noprofht}
Suppose that $d = 1$ and that \eqref{heavy tail band} and \eqref{assumption on h} hold. Suppose moreover that $N \ll W^{1 + 1/2\beta}$ and $\eta \gg (N/W)^\beta / W$. Then we have
\be
   \Lambda^2 \;\prec\; \max \hbb{\frac{\eta^{1/\beta - 1} + 1}{W} , \frac{1}{N \eta}}
\ee
for $z\in \f S$.
\end{theorem}

\begin{corollary}[Eigenvector delocalization]\label{cor:noprofht}
Suppose that $d = 1$ and that \eqref{heavy tail band} and \eqref{assumption on h} hold. If $N \ll W^{1 + 1/2\beta}$ then the eigenvectors of $H$ are completely delocalized in the sense of Proposition \ref{lemma: simple deloc}.
\end{corollary}

\begin{theorem}[Diffusion profile]\label{lm:withprofht}
Suppose that $d = 1$ and that \eqref{heavy tail band} and \eqref{assumption on h} hold for some $\beta \geq 1$. Suppose that $N \ll W^{1 + 1/2\beta}$ and $(W/N)^\beta \le \eta \le 1$.
Then
\be
 |T_{xy} -\Tdet_{xy}| \;\prec\; \frac{1}{N\eta}\,, \qquad
  \absB{P_x |G_{xy}|^2 - \delta_{xy}|m|^2- \abs{m}^2 \Tdet_{xy}} \;\prec\; \frac{1}{N\eta}
 + \frac{\delta_{xy}}{\sqrt{W}}\,.
\ee
Moreover, the analogues of \eqref{Tfin} and \eqref{Tfin1} hold.
\end{theorem}

\appendix

\section{The deterministic profile}\label{sec:prof}

In this appendix we establish bounds and asymptotics for the deterministic profile $\Tdet_{xy}$.

\subsection{Proof of Proposition \ref{prop:profile}}
We assume $d = 1$ and abbreviate $\bb T \equiv \bb T_N^1$ as well as $P \equiv P_N^1$. We assume without loss of generality that $x \in \bb T$.

First we notice that by the symmetry of $f$ and \eqref{decay of f},
the Fourier transform $\wh f(q) \deq \int_\R \me^{-\ii qx} f(x) \, \rd x$
is a smooth function with a Taylor expansion
\be\label{ftaylor}
   \wh f(q) \;=\; 1 - D_\infty q^2 + q^4 g(q) \qquad (q \in \R)\,,
\ee
where $g$ is a bounded smooth function satisfying $g(q)=g(-q)$.
Clearly, $\wh f$ is real and $\|\wh f\|_\infty\le 1$.
Moreover, we claim
that for any  $\e>0$ there exists an $\e'>0$   such that
\be
  \wh f(q) \;\le\; 1-\e' \qquad \text{if} \qquad  |q|\;\ge\; \e\,;
\label{fq}
\ee
indeed, this follows easily from the identity
$$
  1- \wh f(q) \;=\; \int_\R \pb{1 - \cos(qx)} f(x) \,\rd x\,. 
$$

Next, define the lattice
\begin{equation*}
Q \;\deq\; W P \;=\; \frac{2 \pi W}{N} \bb T\,.
\end{equation*}
For $q \in [-\pi W, \pi W]$ define
$$
  \wh S_W(q) \;\deq\; \wh S(q/W)\,.
$$
Thus we get
\be
\pbb{\frac{S}{1-|m|^2 S}}_{x0}
\;=\; 
\frac{1}{N}\sum_{p\in P} 
\me^{\ii p x} \frac{\wh S(p)}{1-|m|^2 \wh S(p)}
\;=\;  \frac{1}{N}\sum_{q\in Q} 
\me^{\ii q x/W} \frac{\wh S_W(q) }{1-|m|^2 \wh S_W(q)}
\label{Ssum1app}
\ee
for $x \in \bb T$.

As a guide for intuition, we have $\wh S_W(q) \approx \wh f(q)$, as can be seen from 
\begin{equation} \label{S approx f}
\wh S_W(q) \;=\; \sum_{x \in \bb T} \me^{-\ii x q / W} \frac{1}{Z_{N,W}} f \pbb{\frac{x}{W}} \;\approx\; \int_{-N/2}^{N/2} \me^{-\ii x q / W} f \pbb{\frac{x}{W}} \frac{1}{W} \, \dd x \;=\; \int_{-N/2W}^{N/2W} \me^{-\ii q y} f(y) \, \dd y \;\approx\; \wh f(q)\,.
\end{equation}
Thus, our proof consists in controlling the error in the approximation
$$
  \Big(\frac{S}{1-|m|^2 S}\Big)_{x0}\;\approx\;
   \frac{1}{W} \int_\bR 
\me^{\ii qx/W} \frac{\wh f(q) }{1-|m|^2\wh f(q)}\rd q\,.
$$

As a first step, we establish basic properties of $\wh S_W$ that are analogous to \eqref{fq} and \eqref{ftaylor}.

\begin{lemma}\label{lm:SW} The function $\wh S_W$ is smooth with uniformly bounded derivatives, real, and symmetric
with $\abs{\wh S_W(q)} \le 1$ and $\wh S_W(0)=1$. Moreover, it has the following properties.
\begin{enumerate}
\item
For any $\e>0$ there exists an $\e'>0$ such that 
\be
  \wh S_W(q)\le 1-\e' \qquad \text{if} \qquad |q| \;\ge\; \e
\label{Sq}
\ee
for large enough $W$ (depending on $\e$).
\item
For any $k,K \in \N$ there is a constant $C_{k,K}$ such that
\be\label{SWdecay}
|\partial_q^k \wh S_W(q)| \;\le\; \frac{C_{k,K}}{\langle q\rangle^K} + O(N^{-1})\,, \qquad (q \in [-\pi W, \pi W])\,.
\ee
\item
There exists a smooth function $g_W$ whose derivatives are bounded uniformly in $W$ such that
\be
    \wh S_W(q) = 1- D q^2 + q^4 g_W(q)\,, \qquad (q \in [-\pi W, \pi W])\,.
\label{SWF}
\ee
\end{enumerate}
\end{lemma}

\begin{proof}
The proof of  (i) is similar to that of \eqref{fq}.

To prove (ii), we use summation by parts combined with \eqref{decay of f}. Let the integers $N_-$ and $N_+$ denote the end points of $\bb T=[N/2, N/2)$, i.e.\ $\bb T = [N_-, N_+]$. Then we find
\begin{align}
\wh S_W(q) &\;=\; \frac{1}{Z_{N,W}}
\sum_{x=N_-}^{N_+} \me^{-\ii qx/W} f \pbb{\frac{x}{W}}
\notag \\
&\;=\; \frac{1}{Z_{N,W}} \sum_{x = N_-}^{N_+ - 1} \sum_{y = N_-}^x \me^{-\ii q y/W} \qbb{f\pbb{\frac{x}{W}} - f \pbb{\frac{x+1}{W}}} + \frac{1}{Z_{N,W}} f \pbb{\frac{N_+}{W}} \sum_{y = N_-}^{N_+} \me^{-\ii q y / W}
\notag \\
&\;=\; \frac{1}{Z_{N,W}} \sum_{x = N_-}^{N_+ - 1} \frac{\me^{-\ii q N_- / W} - \me^{-\ii q (x+1) / W}}{1 - \me^{-\ii q / W}} \qbb{f\pbb{\frac{x}{W}} - f \pbb{\frac{x+1}{W}}} + O_K \pbb{\frac{N}{W} \pbb{\frac{W}{N}}^K}
\notag \\
&\;\leq\; \frac{C_K}{W} \sum_x \frac{1}{\abs{1 - \me^{-\ii q / W}}} \frac{1}{W} \avgbb{\frac{x}{W}}^{-K} +  C_K \pbb{\frac{W}{N}}^{K - 1}
\notag \\ \label{SW}
&\;\leq\; \frac{C}{\abs{q}} + O(N^{-1})\,.
\end{align}
Here we used that
$$
   \frac{1}{\abs{1 - \me^{- \ii q/W}}} \;\le\; \frac{CW}{|q|}
$$
(recall that $\abs{q} \leq \pi W$) and that 
$$ 
\absbb{ f\Big(\frac{x}{W}\Big) - f\Big(\frac{x+1}{W}\Big) } \;\le\; \frac{C_K}{W}
  \Big\langle \frac{x}{W}\Big\rangle^{-K}
$$ 
to estimate the main term, and \eqref{lower bound on W} to estimate
the error term with $K\sim 1/\delta$. We also have the trivial bound on 
$|\wh S_W|\le 1$. Thus we have
$$
   |\wh S_W(q)| \;\le\; \frac{C}{\langle q\rangle} + O(N^{-1})\,.
$$
We can iterate the above argument for the main term in \eqref{SW}, thus obtaining higher order divided differences of $f$. Since $f$ is smooth and decays rapidly, \eqref{SWdecay} follows 
for $k=0$. The proof for $k>0$ is analogous.

In order to prove (iii), we write
\begin{equation*}
g_W(q) \;\deq\; \frac{1 - D q^2 - \wh S_W(q)}{q^4} \;=\; \frac{1}{Z_{N,W}} \sum_{x \in \bb T} \frac{1 - (qx/W)^2/2 - \cos(qx/W)}{(qx/W)^4}  \pbb{\frac{x}{W}}^4 f \pbb{\frac{x}{W}}\,.
\end{equation*}
Now (iii) follows from the fact that the function $h(q) \deq \pb{1 - q^2/2 -\cos(q)} q^{-4}$ is smooth and its derivatives are bounded.
\end{proof}

Having proved Lemma \ref{lm:SW}, we may now complete the proof of Proposition \ref{prop:profile}.
Fix a small constant $\e>0$ and introduce a partition of unity $\chi+\ol\chi = 1$ on $\bR$ 
with smooth functions such that $\chi(q) = 1$ for $|q|\le \e$
and $\chi(q) = 0$ for $|q|\ge 2\e$.
Since $\wh S_W(q)\le 1$ and $|m|^2\le 1- c\eta$ (see \eqref{m3}), 
we have
$$
  1- |m|^2 \wh S_W(q) \;\ge\; 1-|m|^2 \;\ge\; c\eta
$$
with some positive constant $c$. By  \eqref{Sq} we have on the support of $\ol\chi$
\begin{equation} \label{lower bound for SW}
  1- |m|^2 \wh S_W(q) \;\ge\;  \e'
\end{equation}
for some $\epsilon'$ depending on $\epsilon$.
Then from \eqref{Ssum1app} we have
\begin{align}
\pbb{\frac{S}{1-|m|^2 S}}_{x0}
 &\;=\;\frac{1}{N}\sum_{q\in Q}  \me^{\ii qx/W} \frac{\wh S_W(q) \chi(q)}{1-|m|^2 \wh S_W(q)}  
  + \frac{1}{N}\sum_{q\in Q}  \me^{\ii qx/W} \frac{\wh S_W(q) \ol\chi(q)}{1-|m|^2 \wh S_W(q)}  
\notag  \\ \label{FTT}
 &\;=\;\frac{1}{N}\sum_{q\in Q}  \me^{\ii qx/W} \frac{\wh S_W(q) \chi(q)}{1-|m|^2 \wh S_W(q)} +  
 O_K \pbb{\frac{1}{W} \avgbb{ \frac{x}{W}}^{-K}}\,.
\end{align}
Here we used that the function
$$
  R(q)= \frac{\wh S_W(q) \ol\chi(q)}{1-|m|^2 \wh S_W(q)}\,,
$$
extended to the whole real line, is smooth and its derivatives are bounded uniformly in $N$ and $W$ (by \eqref{lower bound for SW}). (These bounds may of course depend on $\epsilon$). Moreover, $R(q)=O_K(\langle q\rangle^{-K})$
for any $K$;
see \eqref{SWdecay}.
By summation by parts, as in \eqref{SW}, we find that for such a function we have
\be\label{Rdecay}
  \frac{1}{N} \sum_{q\in Q}  \me^{\ii q x/W} R(q) \;=\ O_K \pbb{ \frac{1}{W}
\avgbb{\frac{x}{W}}^{-K}}
\ee
for any $K$.

Now we consider the first term in \eqref{FTT}. For the following we use $A_i(q,\eta,N,W)$ with $i = 1,2,3,\dots$ to denote functions that are smooth in $q$ and whose $q$-derivatives are uniformly bounded in $q$, $\eta$, $W$, and $N$.
Using the Taylor expansion \eqref{SWF} and \eqref{m2}, we have (omitting the arguments for brevity)
$$
1-|m|^2 \wh S_W \;=\; \alpha \eta + D  q^2 + A_1 q^4 + A_2 \eta q^2 + A_3 \eta^2 \,.
$$
This gives (again omitting the arguments)
\begin{align}
\pbb{\frac{\wh S_W}{1 - \abs{m}^2 \wh S_W} - \frac{1}{\alpha \eta + D q^2}} \chi &\;=\; \frac{A_4 q^4 + A_5 \eta q^2 + A_6 \eta^2}{\pb{\alpha \eta + D  q^2 + A_1 q^4 + A_2 \eta q^2 + A_3 \eta^2} \pb{\alpha \eta + D q^2}}\chi
\notag \\ 
&\;=\; \frac{A_4 r^4 + A_5 r^2 + A_6}{\pb{\alpha + D r^2 + A_1 \eta r^4 + A_2 \eta r^2 + A_3 \eta} \pb{\alpha + D r^2}}\chi
\notag \\ \label{big frac}
&\;\eqd\; F_{N,W,\eta}(r)\,,
\end{align}
where we introduced the new variable $r \deq \eta^{-1/2} q$. By definition, $A_1, \dots, A_6$ and their $q$-derivatives are uniformly bounded. Since $D \geq c > 0$ and $r \leq \epsilon \eta^{-1/2}$ on the support of $\chi$, we find that for small enough $\epsilon$ the denominator of the second line of \eqref{big frac} is bounded away from zero, uniformly in $r$, $\eta$, $W$, and $N$. We therefore conclude that $F_{N,W,\eta}$ is smooth and its derivatives (in the variable $r$) are uniformly bounded.

Using summation by parts, exactly as in \eqref{SW}, we get
\begin{equation} \label{sum parts 2}
\frac{1}{N}\sum_{q\in Q}  \me^{\ii qx/W}  F_{N,W,\eta}(\eta^{-1/2} q) 
\;\leq \;
\frac{C_K}{W} \avgbb{ \frac{\sqrt{\eta}x}{W} }^{-K}\,.
\end{equation}
Here we used that the sum on the left-hand side ranges over a set of size $O(N/W)$ due to the factor $\chi$ in the definition of $F_{N,W,\eta}$.
Therefore \eqref{big frac} and \eqref{sum parts 2} imply that the first term of \eqref{FTT} is given by
\be\label{1st}
\frac{1}{N}\sum_{q\in Q}  \me^{\ii qx/W} \frac{\wh S_W(q) \chi(q)}{1-|m|^2 \wh S_W(q)}  
\;=\;  \frac{1}{N}\sum_{q\in Q}  \me^{\ii qx/W} \frac{\chi(q)}{\al\eta + D q^2} 
+ O_K \pbb{\frac{1}{W} \avgbb{ \frac{\sqrt{\eta}x}{W} }^{-K}}\,.
\ee
Notice that the error term in \eqref{Rdecay} is smaller than in
\eqref{1st}.
Next, we remove the factor $\chi$ from the main term, exactly as in \eqref{Rdecay}. Plugging this into \eqref{FTT} yields
\be\label{1st1}
 \pbb{\frac{S}{1-|m|^2 S}}_{x0}
  \;=\;  \frac{1}{N}\sum_{q\in Q}  \me^{\ii qx/W} \frac{1}{\al\eta + D q^2} 
+ O_K \pbb{\frac{1}{W} \avgbb{ \frac{\sqrt{\eta}x}{W} }^{-K}}\,.
\ee
We can extend the summation in the main term
$$
  \frac{1}{N}\sum_{q\in Q}  \me^{\ii qx/W} \frac{1}{\al\eta + D q^2}
  \;=\;  \frac{1}{N}\sum_{q\in \frac{2\pi W}{N}\bZ }  \me^{\ii qx/W} \frac{1}{\al\eta + D q^2}
  + O \pbb{\frac{1}{W}\int_\bR  \frac{\ind{\abs{q}\ge \pi W} }{\al\eta + D q^2}\, \rd q}\,,
$$
where the error term on the right-hand side is of order $O(W^{-2})$.
Thus we have
$$
\Big(\frac{S}{1-|m|^2 S}\Big)_{x0}
\;=\; \frac{1}{N}\sum_{q\in  \frac{2\pi W}{N}\bZ}  \me^{\ii qx/W} \frac{1}{\al\eta + D q^2}
 +  O_K \pbb{\frac{1}{W} \avgbb{ \frac{\sqrt{\eta}x}{W} }^{-K}}   + 
O \pbb{\frac{1}{W^2}}\,.
$$
The main term can be computed by the Poisson summation formula
\begin{equation*}
\frac{1}{\pi}\sum_{n\in \bZ} \frac{\me^{\ii nx}}{a^2+n^2} \;=\; \frac{1}{a}\sum_{k\in \bZ} \me^{-a|x+2\pi k|}\,,
\end{equation*}
where $a > 0$.
Thus,
$$
  \frac{1}{N}\sum_{q\in \frac{2\pi W}{N}\bZ}  \me^{\ii qx/W} \frac{1}{\al\eta + D q^2}
  \;=\; \frac{1}{2W\sqrt{D\al\eta}}\sum_{k\in \bZ} \exp \qbb{-\frac{\sqrt{\al\eta}}{W\sqrt{D}}|x+kN|}\,.
$$
This concludes the proof of \eqref{prof1}.

In order to prove \eqref{supexpl}, it suffices to analyse the asymptotics of the expression
\begin{equation*}
R \;\deq\; \frac{1}{W \sqrt{\eta}} \sum_{k \in \Z} \me^{-\sqrt{\eta}\frac{N}{W} k}\,.
\end{equation*}
We consider two cases. If $\eta \geq \pb{\frac{W}{N}}^2$ then $R \asymp \frac{1}{W \sqrt{\eta}}$. On the other hand, if $\eta \leq \pb{\frac{W}{N}}^2$ we use an integral approximation to get
\begin{equation*}
R \;=\; \frac{1}{N \eta} \, \sqrt{\eta}\frac{N}{W} \sum_k \me^{-\sqrt{\eta}\frac{N}{W} k} \;\asymp\; \frac{1}{N \eta}\,.
\end{equation*}
This concludes the proof of \eqref{supexpl}, and hence of Proposition \ref{prop:profile}.

\subsection{Higher dimensions: proofs of Lemmas \ref{lm:prof bound d} and \ref{lm: prof for d}}
\begin{proof}[Proof of Lemma \ref{lm:prof bound d}]
We follow the argument from the proof of Proposition \ref{prop:profile} in the previous section, and merely sketch the differences. We use the $d$-dimensional lattices
\begin{equation*}
\bb T \;\equiv\; \bb T_L^d \,, \qquad Q \;\deq\; \frac{2 \pi W}{L} \bb T\,.
\end{equation*}
Exactly as in \eqref{Ssum1app}, we get
\begin{equation} \label{prof for d}
\pbb{\frac{S}{1 - \abs{m}^2 S}}_{x0} \;=\; \frac{1}{N} \sum_{q \in Q} \me^{\ii q \sdot x / W} \frac{\wh S_W(q)}{1 - \abs{m}^2 \wh S_W(q)} \qquad \text{where} \qquad \wh S_W(q) \;\deq\; \sum_{x \in \bb T} \me^{-\ii q \sdot x / W} s_{x0}\,.
\end{equation}

Next, the basic properties of $\wh S_W$ listed in Lemma \ref{lm:SW}, and their proofs, carry over verbatim to the higher-dimensional setting. Now \eqref{SWF} reads $\wh S_W(q) = 1 - (q \sdot Dq) (1 + A_2(q))$ where $A_2(q) = O(\abs{q}^2)$ uniformly in $W$. Let $\chi$ be a smooth bump function on $\R^d$, as in the proof of Proposition \ref{prop:profile}. As in \eqref{FTT}, we find
\begin{equation*}
\pbb{\frac{S}{1-|m|^2 S}}_{x0} \;=\; \frac{1}{N}\sum_{q}  \me^{\ii q\sdot x/W} \frac{\wh S_W(q) \chi(q)}{1-|m|^2 \wh S_W(q)} +  
 O_K \pbb{\frac{1}{W^d} \avgbb{ \frac{x}{W}}^{-K}}
\end{equation*}
for arbitrary $K \in \N$. Here, and for the rest of this proof, the summation in $q$ ranges over the lattice $\big(\frac{2 \pi W}{N} \Z\big)^d$. We split
\begin{equation} \label{splitting for d}
\pbb{\frac{S}{1 - \abs{m}^2 S}}_{x0} \;=\; \frac{1}{N} \sum_{q} \me^{\ii q \sdot x / W} \frac{\chi(q)}{\alpha\eta + q \sdot D q}
+ \frac{1}{N} \sum_{q} \me^{\ii q \sdot x / W} R(q) \chi(q) + O_K \pbb{\frac{1}{W^d} \avgbb{ \frac{x}{W}}^{-K}}\,,
\end{equation}
where
\begin{equation*}
R(q) \;\deq\; \frac{\wh S_W(q)}{1 - \abs{m}^2 \wh S_W(q)} - \frac{1}{\alpha\eta + q \sdot D q}\,.
\end{equation*}
Note that, unlike in the proof of Proposition \ref{prop:profile}, we keep the cutoff function $\chi$ in the main term since the function $(\eta + q\sdot Dq)^{-1}$ is not integrable in higher dimensions.

The main term of \eqref{splitting for d} can be computed using Poisson summation:
\begin{equation} \label{main term Poi d}
\frac{1}{N} \sum_{q \in ( \frac{2 \pi W}{N} \Z)^d} \me^{\ii q \sdot x / W} \frac{\chi(q)}{\alpha \eta + q D q} \;=\; \frac{(\alpha \eta)^{d/2 - 1}}{W^d \sqrt{\det D}} \sum_{k \in \Z^d}  \pb{V * \varphi_{\sqrt{\alpha \eta}}}\pbb{\frac{\sqrt{\alpha \eta}}{W}D^{-1/2} (x - y + kL)}\,.
\end{equation}
Using that $V(x) \asymp \abs{x}^{2 - d}$ near the origin, we find $\norm{V * \varphi_{\sqrt{\alpha \eta}}}_\infty \leq C W^{-d}$. By treating the two cases $\eta \leq \pb{\frac{W}{N}}^2$ and $\eta \geq \pb{\frac{W}{N}}^2$ separately, we find exactly as in the last paragraph of the proof of Proposition \ref{prop:profile} that \eqref{main term Poi d} is bounded by $C W^{-d} + C (N \eta)^{-1}$.

What remains therefore is the estimate of the error term containing $R$ in \eqref{splitting for d}. To that end, we write
\begin{equation} \label{expansion of R}
R(q) \;=\; \frac{B_4 + \eta B_2 + \eta^2 B_0}{\pb{\alpha\eta + (q\sdot Dq) (1 + A_2) + \eta A_2' + \eta^2 A_0} (\alpha\eta + q\sdot Dq)}\,,
\end{equation}
where $B_0, B_2, B_4, A_0, A_2, A_2'$ are smooth and bounded functions of $q$, each of order $O(\abs{q}^i)$ near the origin
uniformly in $W$ and $\eta$, where $i$ denotes the subscript of the corresponding function. Using the change of variables $q = \sqrt{\eta} \, r$  it is now easy to see that the error term containing $R$ in \eqref{splitting for d} is bounded by $C W^{-d}$. This concludes the proof of Lemma \ref{lm:prof bound d}.
\end{proof}

\begin{proof}[Proof of Lemma \ref{lm: prof for d}]
We need a more precise bound on the error term of \eqref{splitting for d} than the bound $C W^{-d}$ from the proof of Lemma \ref{lm:prof bound d}. In fact, we claim that
\begin{equation} \label{nasty error}
\frac{1}{N} \sum_{q} \me^{\ii q x / W} R(q) \chi(q) \;=\; O_K \pBB{\frac{1}{W^d} \avgbb{\frac{x}{W}}^{-K} + \frac{\eta^{d/2}}{W^d} \avgbb{\frac{\sqrt{\eta} x}{W}}^{-K}}\,.
\end{equation}
The proof of \eqref{nasty error} is a rather laborious exercise in Taylor expansion whose details we omit. The basic strategy is similar to the analysis of \eqref{big frac}, except that we expand $\wh S_W$ up to order $d/2 + 2$ (instead of $4$). This completes the proof of Lemma \ref{lm: prof for d}.
\end{proof}

\subsection{Slowly decaying band: proof of Lemma \ref{lm:ht profile bound} and Proposition \ref{lm:profile heavy tail}}
We begin by proving the following auxiliary result, which gives the relevant asymptotics of $\wh S_W$. For $q \neq 0$ define
\begin{equation} \label{Theta theta}
b(q) \;\deq\; h_0\frac{W}{Z} \int_\bR \dd u \, \frac{1 - \cos u}{\abs{u}^{1 + \beta}} \sigma\pbb{\frac{Wu}{qN}} \;=\; B + O \qbb{\pbb{\frac{W}{qN}}^\beta}\,.
\end{equation}
We also set $b(0) \deq 0$, so that $b$ is continuous.

\begin{lemma} \label{lm:SW heavy tail}
Suppose that $d = 1$ and that \eqref{heavy tail band} and \eqref{assumption on h} hold. Then the following are true.
\begin{enumerate}
\item
For any $K \in \N$ there exists a constant $C_{K}$ such that for $\abs{q} \geq 1$ we have
\begin{equation*}
\abs{ \wh S_W(q)} \;\leq\; \frac{C_{K}}{\abs{q}^K}\,.
\end{equation*}
\item
For $\abs{q} \leq 1$ we have
\begin{equation} \label{asht1}
\wh S_W(q) \;=\; 1 - b(q) \abs{q}^\beta + O(q^2)
\end{equation}
uniformly in $N$ and $W$.
\item
There is a constant $c_1$ such that
\begin{equation*}
\norm{S \ol \Pi} \;\leq\; 1 - c_1 \pbb{\frac{W}{N}}^\beta\,.
\end{equation*}
\end{enumerate}
\end{lemma}

\begin{proof}
Part (i) is proved similarly to \eqref{SWdecay}, using summation by parts.

In order to prove part (ii), we write
\begin{equation} \label{1 - S}
1 - \wh S_W(q) \;=\; \frac{1}{Z} \sum_{x \in \frac{1}{W} \Z} \pb{1 - \cos(qx)} \frac{h(x)}{\abs{x}^{1 + \beta}} \sigma \pbb{\frac{Wx}{N}}\,.
\end{equation}
Let $\chi$ be a smooth, symmetric bump function satisfying $\chi(x) = 1$ for $\abs{x} \leq 1$ and $\chi(x) = 0$ for $\abs{x} \geq 2$. Write $\ol \chi \deq 1 - \chi$. We introduce the splitting
\begin{equation*}
h \;=\; \chi h + \ol \chi (h - h_0) +  h_0\ol \chi
\end{equation*}
on the right-hand side of \eqref{1 - S}. It is easy to check that the two first terms give a contribution of order $O(q^2)$. The last term of the splitting gives rise to
\begin{equation}
\frac{h_0}{Z} \sum_{x \in \frac{1}{W}\Z} \pb{1 - \cos(qx)} \frac{\ol \chi(x)}{\abs{x}^{1 + \beta}} \sigma \pbb{\frac{Wx}{N}}
\;=\; h_0 \frac{W}{Z} \int_\R \pb{1 - \cos(qx)} \frac{\ol \chi(x)}{\abs{x}^{1 + \beta}} \sigma \pbb{\frac{Wx}{N}} \, \dd x + O \pbb{\frac{q^2}{W^2} + \frac{1}{N^2}}\,,
\end{equation}
where the last step follows from a mid-point Riemann sum approximation.
Now a change of variables $u = qx$ easily  yields \eqref{asht1}.

Part (iii) follows from part (ii) using an argument similar to \eqref{norm S Pi}.
\end{proof}

\begin{proof}[Proof of Lemma \ref{lm:ht profile bound}]
The claim follows from the bound
\begin{equation*}
\Tdet_{xy} \;\leq\; \frac{C}{W} \frac{W}{N} \sum_{q \in Q} \frac{\wh S_W(q)}{1 - \abs{m}^2 \wh S_W(q)} \;\leq\; \frac{C}{N \eta} + \frac{C}{W}\int_\R \dd q\, \frac{\wh S_W(q) \ind{\abs{q} \geq \eta^{1/\beta}}}{1 - \abs{m}^2 \wh S_W(q)} \;\leq\; \frac{C}{N \eta} + \frac{C (\eta^{1/\beta - 1} + 1)}{W}\,,
\end{equation*}
where the first term is the contribution of the low modes $\abs{q} \leq \eta^{1/\beta}$ and the second term the contribution of the high modes $\abs{q} \geq \eta^{1/\beta}$, which may be replaced with an integral and estimated using Lemma \ref{lm:SW heavy tail}. We omit the details.
\end{proof}

\begin{proof}[Proof of Proposition \ref{lm:profile heavy tail}]
We proceed similarly to the proof of Proposition \ref{prop:profile}. We choose a cutoff scale $\epsilon$, and denote by $\chi$ the bump function from the proof of Lemma \ref{lm:SW heavy tail}. The scale $\epsilon$ satisfies $\eta^{1/\beta} \ll \epsilon \ll 1$, and will be chosen by optimizing at the end of the proof.

We use the expansions \eqref{asht1} and \eqref{m2}. Thus we find, as in the proof of Proposition \ref{prop:profile},
\begin{equation} \label{main calc for ht}
\pbb{\frac{S}{1 - \abs{m}^2 S}}_{x0} \;=\; \frac{1}{N} \sum_{q \in Q} \me^{\ii q x / W} \frac{\chi(q / \epsilon)}{\alpha \eta + B \abs{q}^\beta}
+ \frac{1}{N} \sum_{q \in Q} \me^{\ii q x / W} R(q) \chi(q/\epsilon) + O\pbb{\frac{\epsilon^{1 - \beta}}{W}}\,,
\end{equation}
where $\chi$ is a smooth bump function as in the proof of Proposition \ref{prop:profile} and
\begin{equation*}
R(q) \;\deq\; \frac{\wh S_W(q)}{1 - \abs{m}^2 \wh S_W(q)} -\frac{1}{\alpha \eta + B \abs{q}^\beta}
\;=\; \frac{(B - b) \abs{q}^\beta + O\pb{\eta^2 + \eta \abs{q}^\beta + q^2}}{\pb{\alpha \eta + b \abs{q}^\beta + O(\eta^2 + \eta \abs{q}^\beta + q^2)}\pb{\alpha \eta + B \abs{q}^\beta}}\,.
\end{equation*}
Note that for $q \in Q \setminus \{0\}$ we have $b(q) \geq c$. Using \eqref{Theta theta} we may therefore estimate, as in \eqref{big frac}, to get
\begin{equation*}
\frac{1}{N} \sum_{q \in Q} \abs{R(q)} \chi(q /\epsilon) \;\leq\; \frac{C}{W} \pbb{ \log N \, \eta^{1 / \beta - 2}  \pbb{\frac{W}{N}}^\beta + 1 + \epsilon^{2 - \beta} \eta^{1/\beta - 1}}\,.
\end{equation*}
Next, the bump function in the main term of \eqref{main calc for ht} may be easily removed, and the summation in $q$ extended to the whole lattice $\frac{2 \pi W}{N} \Z$, at the expense of an error of order $O(\epsilon^{1-\beta}/W)$.  Putting everything together, we get
\begin{equation*}
\pbb{\frac{S}{1 - \abs{m}^2 S}}_{x0} \;=\; \frac{1}{N} \sum_{q \in \frac{2 \pi W}{N} \Z } \me^{\ii q x / W} \frac{1}{\alpha \eta + B \abs{q}^\beta} + \frac{\eta^{1/\beta - 1}}{W} O \pb{W^{-c} + \epsilon^{2 - \beta} + \epsilon^{1 - \beta} \eta^{1 - 1 / \beta}}
\end{equation*}
for some $c > 0$, where we used \eqref{prof assump ht}. Setting $\epsilon \deq \eta^{1 - 1/\beta}$ and Poisson summation yields
\begin{equation*}
\Tdet_{x0} \;=\; \frac{|m|^2}{W \alpha \eta} \pbb{\frac{\alpha \eta}{B}}^{1/\beta} \,  \sum_{k\in \bZ}
 V \qbb{\pbb{\frac{\alpha \eta}{B}}^{1/\beta} \, \frac{x+kN}{W}} + O \pbb{\frac{\eta^{1/\beta - 1}}{W^{1 + c}}}\,.
\end{equation*}
Now \eqref{Tdet ht} follows by noting that by \eqref{asymp for V}, under the assumption \eqref{prof assump ht}, only the term $k = 0$ is of leading order.
\end{proof}

\section{Multilinear large deviation estimates}

In this appendix we give a generalization 
of the large deviation estimate of Corollary B.3 \cite{EYY1}.
The proof is simpler and the statement is formulated
under the assumption \eqref{finite moments} instead of the stronger 
subexponential decay assumption. Moreover, since the current proof does not rely on the Burkholder inequality, it is trivially generalizable to arbitrary multilinear estimates.

Throughout the following we consider random variables $X$ satisfying
\begin{equation} \label{cond on X}
\E X \;=\; 0\,, \qquad \E \abs{X}^2 \;=\; 1\,, \qquad \norm{X}_p \;\leq\; \mu_p
\end{equation}
for all $p$ with some $\mu_p$. Here we set $\norm{X}_p \deq \pb{\E \abs{X}^p}^{1/p}$.

\begin{theorem}[Large deviation bounds] \label{thm: LDE}
Let $\pb{X_i^{(N)}}$, $\pb{Y_i^{(N)}}$, $\pb{a_{ij}^{(N)}}$, and $\pb{b_i^{(N)}}$ be independent families of random variables, where $N \in \N$ and $i,j = 1, \dots, N$. Suppose that all entries $X_i^{(N)}$ and $Y_i^{(N)}$ are independent and satisfy \eqref{cond on X}.
\begin{enumerate}
\item
Suppose that $\pb{\sum_i \abs{b_i}^2}^{1/2} \prec \Psi$. Then $\sum_i b_i X_i \prec \Psi$.
\item
Suppose that $\pb{\sum_{i \neq j} \abs{a_{ij}}^2}^{1/2} \prec \Psi$. Then $\sum_{i \neq j} a_{ij} X_i X_j \prec \Psi$.
\item
Suppose that $\pb{\sum_{i,j} \abs{a_{ij}}^2}^{1/2} \prec \Psi$. Then $\sum_{i,j} a_{ij} X_i Y_j \prec \Psi$.
\end{enumerate}
If all of the above random variables depend on an index $u$ and the hypotheses of (i) -- (iii) are uniform in $u$, then so are the conclusions.
\end{theorem}

The rest of this appendix is devoted to the proof of Theorem \ref{thm: LDE}. Our proof in fact generalizes trivially to arbitrary multilinear estimates for quantities of the form $\sum_{i_1, \dots, i_k}^* a_{i_1 \dots i_k}(u) X_{i_1}(u) \cdots X_{i_k}(u)$, where the star indicates that the summation indices are constrained to be distinct.

We first recall the following  version of the Marcinkiewicz-Zygmund inequality.
\begin{lemma}\label{lm:MZ}
Let $X_1, \dots, X_N$ be a family of independent random variables each satisfying \eqref{cond on X} and suppose that the family $(b_i)$ is deterministic. Then
\be\label{MZ}
\normbb{\sum_i b_i X_i}_p \;\leq\; (Cp)^{1/2} \mu_p \pbb{\sum_i \abs{b_i}^2}^{1/2}
\ee
\end{lemma}

\begin{proof}
The proof is a simple application of Jensen's inequality. Writing $B^2 \deq \sum_{j} \abs{b_i}^2$,
we get, by the classical Marcinkiewicz-Zygmund inequality \cite{stroock} in the first line, that
\begin{align*}
\normbb{\sum_i b_i X_i}_p^p &\;\leq\; (C p)^{p/2} \, \normbb{\pbb{\sum_i \abs{b_i}^2 \abs{X_i}^2}^{1/2}}_p^{p}
\\
&\;=\; (Cp)^{p/2} B^p \, \E \qbb{\pbb{\sum_i \frac{\abs{b_i}^2}{B^2} \abs{X_i}^2}^{p/2}}
\\
&\;\leq\; (C p)^{p/2} B^p \, \E \qbb{\sum_i \frac{\abs{b_i}^2}{B^2} \abs{X_i}^p}
\\
&\;\leq\; (Cp)^{p/2} B^p \mu_p^p\,.
\end{align*}
\end{proof}

Next, we prove the following intermediate result.

\begin{lemma}\label{lm:aXY}
Let $X_1, \dots, X_N, Y_1, \dots, Y_N$ be independent random variables each satisfying \eqref{cond on X}, and suppose that the family $(a_{ij})$ is deterministic. Then for all $p \geq 2$ we have
\begin{equation*}
\normbb{\sum_{i,j} a_{ij} X_i Y_j}_p \;\leq\; C p \, \mu_p^2 \pbb{\sum_{i,j} \abs{a_{ij}}^2}^{1/2}\,.
\end{equation*}
\end{lemma}
\begin{proof}
Write
\begin{equation*}
\sum_{i,j} a_{ij} X_i Y_j \;=\; \sum_j b_j Y_j \,, \qquad b_j \;\deq\; \sum_i a_{ij} X_i\,.
\end{equation*}
Note that $(b_j)$ and $(Y_j)$ are independent families. By conditioning on the family $(b_j)$, we therefore get from Lemma \ref{lm:MZ}
and the triangle inequality that
\begin{equation*}
\normbb{\sum_j b_j Y_j}_p \;\leq\; (C p)^{1/2} \, \mu_p \normbb{\sum_j \abs{b_j}^2}_{p/2}^{1/2} \;\leq\; (C p)^{1/2} \, \mu_p \pbb{\sum_j \norm{b_j}_p^2}^{1/2}\,.
\end{equation*}
Using Lemma \ref{lm:MZ} again, we have
\begin{equation*}
\norm{b_j}_p \;\leq\;  (C p)^{1/2} \, \mu_p \pbb{\sum_i \abs{a_{ij}}^2}^{1/2}\,.
\end{equation*}
This concludes the proof.
\end{proof}

\begin{lemma}\label{lm:aXX}
Let $X_1, \dots, X_N$ be independent random variables each satisfying \eqref{cond on X}, and suppose that the family $(a_{ij})$ is deterministic.
Then we have
\begin{equation*}
\normbb{\sum_{i \neq j} a_{ij} X_i X_j}_p \;\leq\; C p \, \mu_p^2  \pbb{\sum_{i \neq j} \abs{a_{ij}}^2}^{1/2}\,.
\end{equation*}
\end{lemma}

\begin{proof}
The proof relies on the identity (valid for $i \neq j$)
\begin{equation} \label{identity for LDE}
1 \;=\; \frac{1}{Z_N} \sum_{I \sqcup J = \N_N} \ind{i \in I} \ind{j \in J}\,,
\end{equation}
where the sum ranges over all partitions of $\N_N = \{1, \dots, N\}$ into two sets $I$ and $J$, and $Z_N \deq 2^{N - 2}$ is independent of $i$ and $j$. Moreover, we have
\begin{equation} \label{comb bound}
\sum_{I \sqcup J = \N_N} 1 \;=\; 2^{N} - 2\,,
\end{equation}
where the sum ranges over nonempty subsets $I$ and $J$.
Now we may estimate
\begin{equation*}
\normbb{\sum_{i \neq j} a_{ij} X_i X_j}_p \;\leq\; \frac{1}{Z_N} \sum_{I \sqcup J = \N_N} \normbb{\sum_{i \in I} \sum_{j \in J} a_{ij} X_i X_j}_p \;\leq\; \frac{1}{Z_N} \sum_{I \sqcup J = \N_N} C p \, \mu_p^2  \pbb{\sum_{i \neq j} \abs{a_{ij}}^2}^{1/2}\,,
\end{equation*}
where we used that, for any partition $I \sqcup J = \N_N$, the families $(X_i)_{i \in I}$ and $(X_j)_{j \in J}$ are independent, and hence the Lemma \ref{lm:aXY} is applicable. The claim now follows from \eqref{comb bound}.
\end{proof}

As remarked above, the proof of Lemma \ref{lm:aXX} may be easily extended to multilinear expressions of the form $\sum_{i_1, \dots, i_k}^* a_{i_1 \dots i_k} X_{i_1} \cdots X_{i_k}$.
 
We may now complete the proof of Theorem \ref{thm: LDE}.

\begin{proof}[Proof of Theorem \ref{thm: LDE}]
The proof is a simple application of Chebyshev's inequality. Part (i) follows from Lemma \ref{lm:MZ}, part (ii) from Lemma \ref{lm:aXX}, and part (iii) from Lemma \ref{lm:aXY}. We give the details for part (iii).

For $\epsilon >0$ and $D>0$ we have
\begin{align*}
 \P \qBB{ \absbb{\sum_{i \neq j} a_{ij} X_i X_j} \geq N^\e\Psi}
& \;\leq\; \P\qBB{ \absbb{\sum_{i \neq j} a_{ij} X_i X_j} \geq N^\e\Psi \,,\, 
 \pbb{\sum_{i \neq j} \abs{a_{ij}}^2}^{1/2}\leq N^{\e/2}\Psi}
\\
&\qquad + \P \qBB{\pbb{\sum_{i \neq j} \abs{a_{ij}}^2}^{1/2} \geq N^{\e/2}\Psi} \\
&\;\leq\;
  \P \qBB{\absbb{ \sum_{i \neq j} a_{ij} X_i X_j} \geq N^{\e/2} 
\pbb{\sum_{i \neq j} \abs{a_{ij}}^2}^{1/2}} + N^{-D-1}
\\
&\;\leq\; \pbb{\frac{Cp \mu_p^2}{N^{\e/2}}}^p + N^{-D-1}
\end{align*}
for arbitrary $D$. In the second step we used the definition of $\pb{\sum_{i\ne j} |a_{ij}|^2}^{1/2} \prec \Psi$ with parameters $\e/2$ and $D+1$.
In the last step we used Lemma \ref{lm:aXX} by conditioning on $(a_{ij})$. Given $\e$ and $D$, there is a large enough $p$ such that the first term on the last line is bounded by $N^{-D - 1}$. Since $\epsilon$ and $D$ were arbitrary, the proof is complete.

The claimed uniformity in $u$ in the case that $a_{ij}$ and $X_i$ depend on an index $u$ also follows from the above estimate.
\end{proof}

\providecommand{\bysame}{\leavevmode\hbox to3em{\hrulefill}\thinspace}
\providecommand{\MR}{\relax\ifhmode\unskip\space\fi MR }
% \MRhref is called by the amsart/book/proc definition of \MR.
\providecommand{\MRhref}[2]{%
  \href{http://www.ams.org/mathscinet-getitem?mr=#1}{#2}
}
\providecommand{\href}[2]{#2}

\end{document}